\documentclass[letterpaper]{amsart}
\usepackage{tikz-cd}
\usepackage{amsmath}
\usepackage{mathtools}
\usepackage{amsmath,amssymb,amsthm,mathrsfs}
\usepackage[utf8]{inputenc}
\usepackage{hyperref}
\usepackage{enumitem}

\numberwithin{equation}{subsection}
\theoremstyle{plain} 
\newtheorem{theorem}{Theorem}
\numberwithin{theorem}{subsection}
\newtheorem*{theorem*}{Theorem}
\newtheorem*{notationsimplyfy}{Simplification of Notations}
\newtheorem{prop}[theorem]{Proposition}
\newtheorem{lemma}[theorem]{Lemma}

\theoremstyle{definition}
\newtheorem{definition}[theorem]{Definition}

\theoremstyle{remark}
\newtheorem{remark}[theorem]{Remark}
\theoremstyle{definition}

\newcommand{\bbC}{\mathbb{C}}

\newcommand{\Q}{\mathbb{Q}}

\newcommand{\Z}{\mathbb{Z}}

\newcommand{\cA}{\mathcal{A}}
\newcommand{\cB}{\mathcal{B}}
\newcommand{\cC}{\mathcal{C}}

\newcommand{\cF}{\mathcal{F}}

\newcommand{\cH}{\mathcal{H}}

\newcommand{\cK}{\mathcal{K}}
\newcommand{\cL}{\mathcal{L}}
\newcommand{\cM}{\mathcal{M}}
\newcommand{\cN}{\mathcal{N}}
\newcommand{\cO}{\mathscr{O}}

\newcommand{\cQ}{\mathcal{Q}}

\newcommand{\cT}{\mathcal{T}}

\newcommand{\rarr}{\longrightarrow}

\newcommand{\sF}{\mathscr{F}}

\newcommand{\fM}{\mathfrak{M}}

\newcommand{\sO}{\mathscr{O}}
\newcommand{\sA}{\mathscr{A}}

\newcommand{\sD}{\mathscr{D}}

\newcommand{\Var}{\textup{Var}}
\newcommand{\Sym}{\textup{Sym}}
\newcommand{\Ker}{\textup{Ker}}
\newcommand{\DR}{\textup{DR}}
\newcommand{\Gr}{\textup{Gr}}

\DeclareMathOperator{\codim} {codim}

\DeclareMathOperator{\red}{red}

\mathchardef\ordinarycolon\mathcode`\:
\mathcode`\:=\string"8000
\begingroup \catcode `\:=\active
  \gdef:{\mathrel{\mathop\ordinarycolon}}
\endgroup

\title{Isotriviality of smooth families of varieties of general type}

\author{Chuanhao Wei}
\address{Department of Mathematics, Stony Brook University, 100 Nicolls Rd, Stony Brook, NY 11794, USA}
\email{\tt chuanhao.wei@stonybrook.edu}

\author{Lei Wu}
\address{Department of Mathematics, University of Utah,
155 S 1400 E,  Salt Lake City, UT 84112, USA}
\email{\tt lwu@math.utah.edu}

\begin{document}

\begin{abstract}
In this paper, we proved that a log smooth family of log general type klt pairs with a special (in the sense of Campana) quasi-projective base is isotrivial. As a consequence, we proved the generalized Kebekus-Kov\'acs conjecture \cite[Conjecture 1.1]{WW19}, for smooth families of general type varieties as well as log smooth families of log canonical pairs of log general type, assuming the existence of relative good minimal models.
\end{abstract}

\maketitle
\section{introduction}

We start with $f_V:(U,D^U)\to V$, a log smooth family of projective log pairs of log general type, over a smooth quasi-projective variety $V$, with the coefficients of $D^U$ are in $[0,1)$. Here, log smoothness of $f_V$ means that $f_V$ and each stratum of $D^U$ is smooth over $V$ (see \S \ref{subsec:3.1}). Due to \cite{BCHM}, $f_V$ has a relative canonical model, denoted by $f_V^c$. The relative canonical model is a stable family by the invariance of log plurigenera \cite[Theorem 4.2]{HMX18}, and hence it induces a moduli map $\mu:V\to \mathfrak{M}$, where $\mathfrak{M}$ is the coarse moduli space of the corresponding stable family (with a fixed set of coefficients) as in \cite[\S6]{KP16}. Note that $\mathfrak{M}$ is a projective variety, \cite[Theorem 1.1]{KP16}. 

Following \cite[Definition 6.16]{KP16}, we denote the variation of $f_V$ by 
$$\Var(f_V)=\Var(f^c_V):=\dim(\mu(V)).$$ 
We say that $f_V$ is of maximal variation if $\Var(f_V)= \dim V$ and that $f_V$ is isotrivial if $\Var(f_V)=0$.
Since fibers of $f_V$ are of log general type, the above definition is compatible with the variation defined in \cite{Vie83} and \cite{Kaw85}. We write the log Kodaira dimension of $V$ by $\bar\kappa(V)=\kappa(Y, E^Y)$, where the latter is defined to be the Iitaka dimension of the log canonical sheaf $\omega_Y(E^Y)$ and  
$Y$ is a projective compactification of $V$ with $E^Y=Y\setminus U$ a normal crossing divisor. We call $(Y, E^Y)$ a log smooth compactification of $V$.
Note that $\bar\kappa(V)$ does not depend on the choice of $(Y, E^Y)$ and it is a birational invariant, that is, $\bar\kappa(V)$ remains the same, if one replaces $V$ by a smooth birational model. In this paper, we prove
\begin{theorem}\label{T:KK conjecture}
With notations above, we have:
\begin{enumerate}
    \item if $\bar\kappa(V)>-\infty$, then $\bar\kappa(V)\geq \Var(f_V),$
    \item if $\bar\kappa(V)=-\infty$, then $\dim V> \Var(f_V)$.
\end{enumerate}
\end{theorem}
When $D^U$ is empty, we particularly have:
\begin{theorem}\label{T:KK conjecture nonpair}
Let $f_V:U\to V$ be a smooth family of projective varieties of general type. Then,
\begin{enumerate}
    \item if $\bar\kappa(V)>-\infty$, then $\bar\kappa(V)\geq \Var(f_V),$
    \item if $\bar\kappa(V)=-\infty$, then $\dim V> \Var(f_V)$.
\end{enumerate}
\end{theorem}

When $f_V$ is a smooth family of canonical polarized varieties, the above statement is conjectured by Kebekus and Kov\'acs \cite[Conjecture 1.6]{KK08}, which is a natural extension of Viehweg's hyperbolicity conjecture \cite{Vie01}. Meanwhile, Campana made a related conjecture for smooth families of canonical polarized varieties \cite[Conjecture 12.19]{Cam} (see also \cite[Conjecture 1.4]{JK11a}), the Isotriviality Conjecture, which implies the Kebekus-Kov\'acs conjecture. 
The Kebekus-Kov\'acs conjecture for  smooth families of canonical polarized varieties is proved by Taji \cite{Taj16} by proving the Isotriviality Conjecture of Campana.

In the special case that $\Var(f_V)=\dim V,$ Campana's isotriviality conjecture is also known as Viehweg's hyperbolicity conjecture. For smooth families of canonically polarized varieties, it was proved by Campana and P\u{a}un \cite{CP15}. Their result has been extended to the case for smooth families of general type varieties by Popa and Schnell using Hodge modules \cite[Theorem A]{PS17}. They \cite[\S4.3]{PS17} implicitly gave an extension of the Kebekus-Kov\'acs conjecture for families with geometric generic fibers admitting good minimal models (particularly for smooth families of general type varieties). They proved in \emph{loc. cit.} a special case, with additionally assuming an abundance-type conjecture of Campana–Peternell. Based on the Hodge-module construction of Popa and Schnell, we further extended Viehweg's hyperbolicity to smooth families of log general type pairs (see \cite[Theorem A(1)]{WW19}). We then extended the Kebekus-Kov\'acs conjecture to the case for smooth families of log general type pairs (see \cite[Conjecture 1.1]{WW19}). 

Theorem \ref{T:KK conjecture} gives a confirmative answer to the generalized Kebekus-Kov\'acs conjecture. 
It is worth mentioning that the generalized Kebekus-Kov\'acs conjecture even for smooth families of general type varieties was not known before. 

There are three reasons why we consider Theorem \ref{T:KK conjecture} (the pair version) instead of Theorem \ref{T:KK conjecture nonpair} (the non-pair version):\\
(1) The statement of Theorem \ref{T:KK conjecture nonpair} is about the log Kodaira dimension of the base, where the log Kodaira dimension is defined by embedding the base into a pair. \\
(2) To prove Theorem \ref{T:KK conjecture nonpair}, we first need to compactify $f_V$
\[
\begin{tikzcd}
U \arrow[r,hook]\arrow[d,"f_V"]& (X,D^X)\arrow[d,"f"]\\
V \arrow[r,hook] & (Y,D^Y),
\end{tikzcd}
\]
and then do stable reductions (see \S \ref{S:Geo const}), where $f$ is a morphism of log smooth pairs.\\
(3) The proof of Theorem \ref{T:KK conjecture nonpair} needs to use Campana's fibrations, but they are naturally defined as morphisms of pairs (see \S \ref{subS:2.1}).

\begin{remark}
Actually, we can consider the case that $f_V$ is a log smooth family of log smooth pairs of log general type with the coefficients of $D^U$ in $[0,1]$, if we assume that $f_V$ admits a relative good minimal model over a Zariski open subset $V_0\subseteq V$. Then, we have the smooth family $f^\epsilon_V: (U,D^U_\epsilon:=(1-\epsilon)D^U)\to V$ of log smooth pairs of log general type with the coefficients of $D^U_\epsilon$ in $[0,1)$ and $\Var(f^\epsilon_V)\ge \Var(f_V)$ by \cite[Lemma 3.1]{WW19}. See \cite[\S3]{WW19} for the definition of variation in the log canonical case. The inequality  $\Var(f^\epsilon_V)\ge \Var(f_V)$ implies that Theorem \ref{T:KK conjecture} still holds for smooth families of log canonical pairs of log general type, assuming the existence of relative good minimal models.
\end{remark}

In order to prove Theorem \ref{T:KK conjecture}, we prove Campana's Isotriviality Conjecture for log smooth families of log general type pairs, Theorem \ref{T:isotri} below. We continue to use the notations in Theorem \ref{T:KK conjecture}. We now fix a log smooth compactification $(Y, E^Y)$ of $V$, such that  the moduli map $\mu:V\to \mathfrak{M}$ extends to $\bar{\mu}:Y\to \mathfrak{M}$. Considering the Stein factorization of $\bar{\mu}$, we get a morphism $s:Y\to M$ with $M$ a normal variety so that $\bar\mu$ factors through $s$ and $s$ has generically connected fibers. In particular, $s:(Y,E^Y)\to M$ is a fibration of $(Y,E^Y)$ over $M$ with $\dim M=\Var(f_V)$ (see Definition \ref{def:Campfib}). 
Campana defined the \emph{fiberations of general type} and \emph{specialty} in the category of geometric orbifolds; see Definition \ref{def:specialty} for a simplified version in our setting.

\begin{theorem}\label{T:isotri}
With notations as above, the fibration $s$ is of general type. 
As a consequence, if $V$ is a special quasi-projective variety, then $f_V$ is birationally isotrivial, i.e. $\Var(f_V)=0$.
\end{theorem}
The first statement of Theorem \ref{T:isotri} implies Theorem \ref{T:KK conjecture}, thanks to the additivity of the log Kodaira dimension for fibrations of general type, \cite[Theorem 6.3]{Ca11}. 

When $f_V$ is a smooth family of canonically polarized variety, the second statement of Theorem \ref{T:isotri} is the Isotriviality Conjecture of Campana as we mentioned above. In the case that $\dim V\leq 3$, it is proved by Jabbusch and Kebekus \cite{JK11a}. Their proof uses Campana's theory of geometric orbifolds and the so-called Viehweg-Zuo sheaves constructed in \cite{VZ02} as well as a refine result on Viehweg-Zuo sheaves in \cite{JK11b}, which roughly asserts that Viehweg-Zuo sheaves factor through the moduli. 
Based on the strategy of Jabbusch and Kebekus and the method of Campana and P\u aun in proving Viehweg's hyperbolicity conjecture, Taji \cite{Taj16} proved the isotriviality conjecture for smooth families of canonically polarized varieties in general.

By using Hodge modules, Popa and Schnell \cite{PS17} constructed the Viehweg-Zuo sheaves for smooth families of general type varieties with maximal variation (or more precisely for smooth families admitting relative good minimal models). 
We further constructed Viehweg-Zuo sheaves for families of pairs with maximal variation in \cite{WW19}. A key step to prove Theorem \ref{T:isotri} is to construct the Viehweg-Zuo sheaves for families of (log) general type varieties  (pairs) with arbitrary variation and prove that they factor through the ``moduli" in the way analoguous to the result of Jabbusch and Kebekus.

\begin{definition}\label{D:def of B}
In regard to a log smooth compactification $(Y,E)$ of $V$ and an extension $\bar\mu: Y\to \mathfrak{M}$ of the moduli map $\mu$, 
we define the subsheaf 
$$\cB_{(Y, E^Y)}\subseteq \Omega_Y(\log E^Y)$$
by the saturation of the image of the natural morphism $\bar \mu^*\Omega_\mathfrak{M}\to \Omega_Y(\log E^Y)$ in $\Omega_Y(\log E^Y)$, where $\Omega_\mathfrak{M}$ is the sheaf of K\"ahler differentials (over $\bbC$) and $\Omega_Y(\log E^Y)$ the sheaf of log 1-forms with log poles along $E^Y$. 
\end{definition}

The above definition follows \cite[Notation 1.2]{JK11b}. However, they only considered the case when $\mathfrak{M}$ is the moduli space of canonically polarised manifolds in \emph{loc. cit}. 
Note that, since $\bar \mu$ factors through $s:Y\to M$, and $M\to \mathfrak{M}$ is quasi-finite, we have that $\cB_{(Y, E^Y)}$ is also the saturation of the image of $s^*\Omega_M \to \Omega_Y(\log E^Y)$, and the rank of $\cB_{(Y, E^Y)}$ is equal to the dimension of $M$, which is also the variation of $f_V$. 
Now, we state the main technical result of this paper.

\begin{theorem}\label{T:refined VZ}
Let $f_V:(U,D^U)\to V$ be a log smooth family of projective log pairs of log general type, over a smooth quasi-projective variety $V$, with the coefficients of $D^U$ in $[0,1)$. If $\Var(f_V)>0$, then after replacing $V$ by a further log resolution, we have a log smooth compactification $(Y,E^Y)$ of $V$, such that there exists an invertible subsheaf $\cA\subset \Sym^{[n]} \cB_{(Y,E^Y)}$, for some positive integer $n$, with $\kappa(\cA)\geq \Var(f_V)$.  
\end{theorem}
We call the invertible sheaf $\cA$ the refined Viehweg-Zuo Sheaf of $f_V$ on $(Y, E^Y)$. 
The proof of Theorem \ref{T:refined VZ} is built upon the refinement of the stable reduction used in \cite{WW19} (see Section \ref{S:refined VZ} for details). Another key input of its proof is the use of Hodge modules. Roughly speaking, we essentially need Saito's decomposition theorem for pure Hodge modules to compare Viehweg-Zuo sheaves before and after base-changes, see Theorem \ref{T:common HM} in Section \ref{S:twisted} for details. 

\subsection*{Structure of the paper}
In Section \ref{S:birational equiv}, we recall some of Campana's definitions and results about orbifold fibrations, and concludes the proof of Theorem \ref{T:isotri} using Theorem \ref{T:refined VZ}. In Section \ref{S:twisted}, we fixed the notations and show some useful results using Saito's theory of Hodge modules. In Section \ref{S:Geo const}, we prove results related to stable reductions, and using them to make the geometric constructions that are needed in Section \ref{S:refined VZ} to construct the refined Viehweg-Zuo Sheaf in our setting.

\subsection*{Acknowledgement}
We would like to thank Christian Schnell for useful discussions during the preparation of the paper. The first author also gets some inspiration from a workshop held in Shanghai Center for Mathematical Sciences.

\section{Birationally equivalent fibrations }\label{S:birational equiv}
\subsection{Birationally equivalent fibration in the sense of Campana}\label{subS:2.1}

We now recall Campana's birational equivalence of fibrations.
We mainly restrict ourselves to the setting of Theorem \ref{T:isotri} for our application while Campana works more generally in the category of geometric orbifolds in \cite{Ca11}. See also \cite{JK11a} for a more approachable introduction.

\begin{definition}\label{def:Campfib}
We say that $s:(Y,E^Y)\to M$ is \emph{a fibration of a log pair $(Y,E^Y)$ over $M$}, if $s$ is a dominate projective morphism with generically connected fibers, 
and $W$ a normal variety. For simplicity we always assume that $Y$ is a smooth quasi-projective variety. 
Given two fibrations $s:(Y,E^Y)\to M$ and $s':(Y',E^{Y'})\to M'$, we say that $s'$ is \emph{dominant} over $s$, if we have the following commutative diagram
\begin{equation*}
    \begin{tikzcd}
    (Y,E^Y)\arrow[d, "s"]& (Y',E^{Y'})\arrow[d, "s'"]\arrow[l, "u"]\\
    M & M'\arrow[l, "v"],
    \end{tikzcd}
\end{equation*}
with both $u$ and $v$ are birational, and $u_*E^{Y'}=E^Y$. We say two fibrations $s$ and $s'$ are \emph{birationally equivalent} if they both can be dominated by a third fibration $s'':(Y'',E^{Y''})\to M''$.
\end{definition}
Using the recipe in \cite[Definition 3.2]{Ca11} (see also \cite[Construction and Definition 5.3]{JK11a}), we obtain the $\cC$-base $(M, \Delta^s)$ associated to the fibration $s$. 
We say that a fibration $s:(Y,E^Y)\to (M, \Delta^s)$ is a \emph{neat} model if $(Y,E^Y)$ and $(M, \Delta^s)$ are log smooth, and for all of the divisors $F$ with $\codim s(F)\geq 2$, $F \subset E^Y$, \cite[Assumption 5.4]{JK11a}. The name of neat model is adopted from \cite[Definition 4.1]{Taj16}, which serves a similar purpose as ``strictement nette et haute'' model in \cite{Ca11}, in the general case. By definition, starting from an arbitrary fibration $s:(Y,E^Y)\to M $, the associated fibration $s:(Y,E^Y)\to (M, \Delta^s)$ over the associated $\cC$-base $(M, \Delta^s)$ is not always neat even when the pairs $(Y,E^Y)$ and $(M, \Delta^s)$ are log smooth. We recall the following result in \cite{Taj16}, which is essentially proved in \cite[Section 10]{JK11a}.
\begin{prop}\cite[Proposition 4.2]{Taj16}\label{prop:neatmodel}
Every fibration $s:(Y,E^Y)\to (M,\Delta^s)$ is dominated by a neat model.
\end{prop}

\begin{definition}(\cite[Definition 4.10, 4.16 and 4.17]{Ca11})\label{def:specialty}
Let $s:(Y,E^Y)\to M$ be a fibration with  $(Y, E^Y)$ log smooth and $M$ projective. Let $\tilde{s}:(\tilde{Y}, E^{\tilde{Y}})\to (\tilde{M}, \Delta)$ be neat model dominant over $s$, with $ (\tilde{M}, \Delta)$ the induced $\cC$-base of $\tilde{s}$.  We define the canonical dimension of $s$ by $\kappa(\tilde{M},\Delta^{\tilde{s}})$. It does not depend on the choice of the neat model $\tilde{s}$ dominant over $s$ (see \cite[Corollaire 4.11]{Ca11}). Then we define:
\begin{enumerate}
    \item a fibration $s$ is \emph{of general type} if its canonical dimension is the same as the dimension of the base,
    \item  $(Y,E^Y)$ is \emph{special} if there exists no fibration $s:(Y,E^Y)\to M$ of general type with $\dim(M)>0$.
\end{enumerate} 
We also say the smooth quasi-projective variety $V=Y\setminus E^Y$ is special if $(Y,E^Y)$ is so. One can easily check that the specialty of $V$ does not depend on the choice of the compactification $(Y,E^Y)$ and that specialty is a birational invariant, that is, $V$ is special if and only if $V'$ is so, where $V'$ is a smooth quasi-projective variety properly birational to $V$.
\end{definition}

\subsection{Using the refined Viehweg-Zuo sheaf to prove the main theorem}

We use Theorem \ref{T:refined VZ} to give a proof of Theorem \ref{T:isotri}. The proof follows the strategy in \cite{JK11a}, which is also used by Taji in proving the Isotriviality conjecture of Campana \cite[Theorem 1.5]{Taj16}. 

\begin{proof}[Proof of Theorem \ref{T:isotri}]
When $f_V$ is birationally isotrivial, the first statement of Theorem \ref{T:isotri} is obvious. We hence assume $\Var(f_V)>0$.
Using Proposition \ref{prop:neatmodel}, we fix a neat model dominating $s$, $\tilde{s}:(\tilde{Y}, E^{\tilde{Y}})\to (\tilde{M}, \Delta^{\tilde{s}})$, . Let $\mu:(\tilde{Y}, E^{\tilde{Y}})\to (Y, E^Y)$ be the induced birational morphism. Without loss of generality, we further assume that $E^{\tilde{Y}}\supset \mu^{-1}E^Y$ by adding more components to $E^{\tilde{Y}}$ (by doing this the $\cC$-base $(\tilde M, \Delta)$ stays the same).

Recall that $\cB_{(Y, E^Y)}$ is defined to be the saturation of the image of the natural morphism $\tilde{s}^*: \Omega_{M}\to \Omega_{Y}(\log E^{Y})$.
We similarly define $\cB_{(\tilde{Y}, E^{\tilde{Y}})}$ to be the saturation of the image of the natural morphism $\tilde{s}^*: \Omega_{\tilde{M}}\to \Omega_{\tilde{Y}}(\log E^{\tilde{Y}})$. By \cite[Proposition 3.3]{JK11a}, we have that $\Sym^{[n]}\cB_{(\tilde{Y}, E^{\tilde{Y}})}$ is also the saturation of the image of the composed natural morphisms
$$\mu^*\Sym^{[n]}\cB_{(Y, E^Y )}\to \mu^*\Sym^{[n]}\Omega_Y(\log E^Y)\to \Sym^{[n]}\Omega_{\tilde{Y}}(\log E^{\tilde{Y}}).$$
Hence, the refined Viehweg-Zuo sheaf $\cA \subset  \Sym^{[n]} \cB_{(Y,E^Y)}$ on $(Y, E^Y)$ lifted to $\mu^* \cA\subset \Sym^{[n]}\cB_{(\tilde{Y}, E^{\tilde{Y}})}$.
Hence, to make notations simple, we can assume that $s:(Y, E^Y)\to (M, \Delta^s)$ itself is neat.

By definition, to prove the fibration $s: (Y,E^Y)\to M$ is of general type, one only needs to show that $(M, \Delta^s)$ is of log general type. 
Since $\cB_{(Y, E^Y )}$ is saturated, $\Sym^{[n]}\cB_{(Y, E^Y )}\subseteq \Sym^{[n]}\Omega_Y(\log E^Y)$ is also a saturated subsheaf by \cite[Proposition 3.3]{JK11a}. Due to \cite[Proposition 5.7]{JK11a}, see \cite[Notation 4.1]{JK11a} for the notations, we have a natural isomorphism 
$$\iota: \Sym^{[n]}_{\cC}\Omega_M(\log \Delta^s)\to s_* \Sym^{[n]}\cB_{(Y, E^Y )}$$
for all $n$. Using Theorem \ref{T:refined VZ}, we have the refined Viehweg-Zuo sheaf $\cA\subset \Sym^{[n]}\cB_{(Y, E^Y )}$, and we can further assume that $\cA$ is saturated. Let $\cA_M$ be the saturation of $s_*\cA$ in $\Sym^{[n]}_\cC \Omega_M(\log \Delta^s)$, and in particular, it is a line bundle with  $\kappa_\cC(\cA_M)=\kappa(\cA)=\dim M$, by \cite[Proposition 5.7, Corollary 5.8]{JK11a}. 
 By applying \cite[Theorem 5.2]{Taj16}, we get that $(M, \Delta^s)$ is of log general type.
\end{proof}

\section{Construction of Higgs sheaves}\label{S:twisted}
In this section, we use Seito's theory of Hodge modules to show some results about Hodge bundles that will be used to construct the Viehweg-Zuo Sheaves. Some constructions are inspired by \cite{PS17}. 
\subsection{Notations and remarks on log smooth morphism}\label{subsec:3.1}
For a log pair $(X,D^X)$, we mean that $X$ is a normal variety with $D^X$ a $\Q$-divisor and the log canonical divisor $K_X+D^X$ is $\Q$-Cartier. We also write $\omega_X(D^X)$ the $\Q$-line bundle given by the $\Q$-Cartier divisor $K_X+D^X$. We follow the terminology of singularities of pairs as in \cite[\S2.3]{KM98}. 

We say that the pair $(X,D^X)$ is log smooth if $X$ is smooth and the support of $D^X$, denoted by $D^X_{\red}$, is normal crossing. 
We denote $\Omega_X(\log D^X):=\Omega_X(\log D^X_{\red})$, the sheaf of log 1-forms with logarithmic poles along $D^X_{\red}$.
Notice that we have used $\Omega_X(\log D^X)$ to denote the sheaf of $\cC$-log forms in the sense of Campana in the proof of Theorem \ref{T:isotri}. However, Campana's $\cC$-sheaves only make their appearance in the proof of Theorem \ref{T:isotri} in Section \ref{S:birational equiv} but not in the rest part of this paper. 

\begin{definition}
We say that $f:(X,D^X)\to (Y,D^Y)$ is a morphism of log smooth pairs, if both $(X, D^X)$ and $(Y, D^Y)$ are log smooth, and $f^{-1}(D^Y):=f^*(D^Y)_{\red}\subset D^X_{\red}$. We say that $f$ is strict if $f^{-1}(D^Y)=D^X_{\red}$.

We say that $f$ (as a morphism of log smooth pairs) is \emph{log smooth} if we further have that $f$ is dominant as a morphism of schemes and the cokernal of the log differential map
$$df: f^*\Omega_Y(\log D^Y)\to \Omega_X(\log D^X)$$
is locally free. In the case that $D^Y$ is empty, we have the following well-known result. For completeness, we give a brief proof here.

\begin{lemma}\label{L:log smooth equivalence}
$f:(X, D^X)\to Y$ being log smooth is equivalent to each stratum (including $X$) of $(X, D^X)$ being smooth over $Y$.
\end{lemma}
\begin{proof}
Fix one component $D^X_1$ of $D^X$. Consider the following commutative diagram, and define $\cA$ and $\cB$ with all horizontal and vertical complexes forming short exact sequences
\begin{equation*}
    \begin{tikzcd}
        f^*\Omega_Y(-D^X_1)\arrow[r]\arrow[d] & f^*\Omega_Y\arrow[r]\arrow[d] & (f^*\Omega_Y)|_{D^X_1}\arrow[d]\\
        \Omega_X(\log D^X)(-D^X_1)\arrow[r]\arrow[d] & \Omega_X(\log D^X-D^X_1)\arrow[r]\arrow[d] & \Omega_{D^X_1}(\log (D^X-D^X_1)|_{D^X_1})\arrow[d]\\
        \Omega_f(-D^X_1)\arrow[r]& \cA \arrow[r]& \cB.
    \end{tikzcd}
\end{equation*}
Working locally on $X$, if each stratum of $(X, D^X)$ being smooth over $Y$, then by induction on the number of components of $D^X$, we can assume that the right two vertical complexes split, and hence the left complex splits which implies that $f$ is log smooth.

On the other hand, consider the following commutative diagram, with all horizontal and vertical complexes forming short exact sequences
\begin{equation*}
    \begin{tikzcd}
        f^*\Omega_Y\arrow[r]\arrow[d] & f^*\Omega_Y\arrow[r]\arrow[d] & 0 \arrow[d]\\
        \Omega_X(\log D^X-D^X_1)\arrow[r]\arrow[d] & \Omega_X(\log D^X)\arrow[r]\arrow[d] & \cO_{D^X_1}\arrow[d]\\
        \cA\arrow[r]&  \Omega_f \arrow[r]& \cO_{D^X_1}.
    \end{tikzcd}
\end{equation*}
If $f$ is log smooth, then locally on $X$ the middle vertical short exact complex splits, which implies the left one splits. This implies the right vertical complex in the first commutative diagram splits. Hence, by induction, each stratum of $(X, D^X)$ is smooth over $Y$.
\end{proof}


In the case that the morphism $f:(X,D^X)\to (Y,D^Y)$ is dominant, 
we further denote by $D_h^X$, the horizontal part of $D^X$, which means it contains all components that are dominant over $Y$, and by $D_v^X:=D^X-D^X_h$, the vertical part of $D^X$. 
In this case, we are always making the following assumption in this pape: \\
(Assumption. 0) $D^Y$, the boundary divisor of the log smooth base, and $D_v^X$, the vertical boundary divisor, are reduced, that is, their coefficients are $1$, and the coefficients of $D_h^X$, the horizontal boundary divisors are always in $(0,1]$. 

Notice that (Assumption. 0) is enough for the proof of Theorem \ref{T:KK conjecture} as the log smooth morphism $f_V$ has no boundary on the base and $D^U$ is dominant over $V$.

Moreover we write by $\lceil f \rceil: (X, \lceil D^X\rceil)\to  (Y,\lceil D^Y\rceil=D^Y)$ the morphism with rounding up the pairs. Hence $f$ is log smooth if and only if $\lceil f\rceil$ is.
\end{definition}

When we write $\Sym^i\sF$ (or $\Sym^{[i]}\sF$ to be constent with notations from references), we always consider the reflexive hull of the $i$-th symmetric power of the coherent sheaf $\sF$; similarly, for $\det \sF$, we also consider the reflexive hull of the determinant of $\sF$. Actually, due to the following important remark, taking the reflexive hull is also not necessary, since it only modifies the sheaf over a closed subset of codimension $\geq 2$. The next remark will be frequently used in later sections.

\begin{remark}\label{R:codim 2}
To prove Theorem \ref{T:refined VZ} we only need to prove it over a \emph{big open subset} of $Y$, a Zariski open subset with its complement of codimension $\geq 2$, so, we can feel free to ignore a \emph{small closed subset} of $Y$, a Zariski closed subset of codimension $\geq 2$. In particular, we can always assume the divisors on $Y$ are smooth. Furthermore, if we have the morphism of log smooth pairs, $f:(X,D^X)\to (Y,D^Y)$, is log smooth outside of the vertical boundary, i.e. $f$ is log smooth over $Y\setminus D^Y$, then by ignoring a small closed subset of $Y$, we have that $f$ itself is log smooth by the following lemma. 
\end{remark}

\begin{lemma}\label{L:log smooth out of small}
Fix a morphism of log smooth pairs, $f:(X,D^X)\to (Y,D^Y)$. Assume that $f|_{X\setminus f^{-1}D^Y}$ is log smooth over $Y\setminus D^Y$, then it is log smooth over a big open subset of $Y$.
\end{lemma}
\begin{proof}
Since we can ignore all components in $D^X$ that maps to a small closed subset of $Y$, we only need to show that, for any closed point $p\in E$, where $E$ is any smooth stratum of $D^X$ that dominates a smooth component $F$ of $D^Y$, the cokernal of the natural morphism 
$df:f^*\Omega_Y(\log D^Y)\to \Omega_X(\log D^X)$
is locally free, on an open neighborhood $U$ of $p$. However, restricted on $U$, the local sections of $\Omega_X(\log D^X)|_U$ are generated by the sections from $\Omega_E(\log D^E)$ and $\{dx_1/x_1,..., dx_k/x_k\}$, where $D^E:=D^X_h|_E$, (noting that $E$ is contained in $D^X_v$,) and $x_i$ are local regular functions that defines $E$. Similarly, locally around $f(p)$, local sections of $\Omega_Y(\log D^Y)$ are generated by the local sections of $\Omega_F$ and $dy/y$, where $y$ is a local function that defines $F$. However, on $U$, we have $f^*y=u\Pi_i x_i^{a_i}$, where $u$ is a local unit, and $a_i$'s are positive integers. Hence we have, $df(f^*dy/y)=\sum_i a_i dx_i/x_i$, which is a local section of $\Omega_X(\log D^X)|_U$ without zero locus. Now we are only left to show that the cokernal of the natural morphism
$f|_E^*\Omega_F\to \Omega_E(\log D^E)$ is locally free, over a Zariski open subset of $F$, which is true by generic smoothness.
\end{proof}

\begin{definition}
Given a log smooth pair $(X,D)$, we denote
\begin{align*}
    \cT_{(X, D)}:&= \Omega_X(\log D)^\vee\\
    \sA^\bullet_{(X, D)}:&=\bigoplus_i\Sym^i\cT_{(X, D)}, 
\end{align*}
that is, $\cT_{(X, D)}$ is the sheaf of (algebraic) vector fields with logarithmic zeroes along $D^X$ and $\sA^\bullet_{(X, D)}$ is the associated graded algebra of symmetric powers. The sheaves $\cT_{(X, D)}$ and $\sO_X$ generate a subalgebra $\sD_{X,D}$ of $\sD_X$, the sheaf of (algebraic) differential operators. We call $\sD_{(X,D)}$ the sheaf of (algebraic) log differential operators. The order filtration $F_\bullet$ on $\sD_X$ induces the order filtration on $\sD_{(X,D)}$. With this filtration, we have a canonical isomorphism 
\[\sA^\bullet_{(X, D)}= \Gr^F_\bullet \sD_{(X, D)}\]
where $\Gr^F_\bullet \sD_{(X, D)}$ denotes the associated graded algebra. 
\end{definition}

\begin{definition}
Given a log smooth morphism of log smooth pairs $f:(X,D^X)\to (Y,D^Y)$, 
\begin{align*}
\Omega_f&:=\text{coker}(f^*\Omega_Y(\log D^Y)\to \Omega_X(\log D^X)), \\
\Omega^i_f&:= \wedge^i \Omega_f,\\
\DR_f&:=[\cO_X\to\Omega^1_f\to ...\to \Omega_f^k],
\end{align*}
with cohomology degrees at $-k,...,0$, $k=\dim X-\dim Y$. We call $\DR_f$ the relative log de Rham complex of $f$.

We set $F_p\Omega^i_f=\Omega^i_f$ for $p\ge 0$ and $F_p \Omega^i_f=0$ for $p<0$. Then we have the induced filtration on the complex $\DR_f$:
$$F_\bullet \DR_f:=[F_{\bullet-k}\cO_X\to F_{\bullet-k+1}\Omega^1_f\to ...\to F_\bullet\Omega_f^k].
$$
The associated graded complexes are $\Gr^F_i \DR_f=\Omega^{k-i}_f[i]$ for all $i$. 

Taking the push-forward functor $Rf_*$ on $(\DR_f,F_\bullet\DR_f)$, we then have the relative log Hodge-to-de Rham spectral squence. See Theorem \ref{thm:logHtdH} for further disucssions. 

Since $D^X$ and $D^Y$ are possibly $\Q$-divisors, we write the $\Q$-line bundle by
$$\omega_f:=\omega_X(D^X)\otimes f^*(\omega^{-1}_Y(-D^Y)).$$
Using $\lceil f\rceil$, we particularly have $\omega_{\lceil f\rceil}=\Omega_f^k$. When $D^X$ and $D^Y$ are empty, $\omega_f=\omega_X\otimes f^*(\omega^{-1}_Y)$, the relative canonical sheaf of $f$.
\end{definition}

Everything defined as above except $\omega_f$, only depends on the the information of $\lceil f\rceil$. We still use the sub-index $f$ to make notations simpler.

\subsection{Direct image of graded $\sA^\bullet$-modules}\label{SS:Direct image of graded log D-modules}
Given a morphism of log smooth pairs $f:(X, D^X)\to (Y, D^Y)$, fixing a line bundle $\cL$ on $X$, we denote 
$$\mathcal{M}^i_{f,\cL ,\bullet}:=\mathbf{R}^if_*(\cL^{-1}\otimes \omega_X(\lceil D^X\rceil)\otimes^\mathbf{L}_{\sA^\bullet_{(X,D^X) }}f^*(\sA^\bullet_{(Y,D^Y)}\otimes \omega^{-1}_Y(-\lceil D^Y\rceil)))$$
for each $i$.
By definition, $\mathcal{M}^i_{f,\cL ,\bullet}$ is a graded $\sA^\bullet_{(Y,D^Y)}$-modules. 
Hence it possesses a Higgs-type morphism $$\theta_{f,\cL }^i:\cM_{f,\cL,\bullet}^i\to \cM_{f,\cL, \bullet+1}^{i}\otimes \Omega_{(Y,D^Y)}.$$
Denote $\mathcal{K}^i_{f,\cL, \bullet}:=\Ker(\theta^i_{f,\cL})$.
By repeatedly composing $\theta_{f,\cL }^i\otimes \textup{Id}$ with itself $k$ times, we have the induced morphism
$$\theta_{f,\cL,\bullet}^{i,k}:\cM_{f,\cL,\bullet}^i\to \cM_{f,\cL, \bullet+k}^{i}\otimes \Sym^k\Omega_{(Y,D^Y)}.$$

The following proposition is the pair analogue of \cite[Lemma 6.2]{WW19}.
\begin{prop}\label{P:iso M to relative DR}
If $f$ is log smooth, we have the following quasi-isomorphism of graded $\sA^\bullet_{(Y,D^Y)}$-modules
\begin{align*}\label{E:relative DR comparison}
    \cM_{f,\cL,\bullet}^i&\stackrel{q.i.}{\simeq}\mathbf{R}^if_*(\cL^{-1}\otimes \Gr_\bullet \DR_f).
\end{align*}
In particular, we have that $\mathcal{M}^i_{f,\cL ,\bullet}$ are coherent over $\cO_Y$. 
\end{prop}
\begin{proof}
First, we have the relative log Spencer complex
\[\sA^\bullet_{(X,D^X)}\otimes \cT^k_f\to\sA^\bullet_{(X,D^X)}\otimes \cT^{k-1}_f\to\dots\to \sA^\bullet_{(X,D^X)},\]
given by 
\[P\otimes v_1\wedge\cdots\wedge v_p\mapsto \sum_i Pv_i\otimes v_1\wedge\cdots\wedge\hat{v}_i\wedge\cdots\wedge v_p,\]
for $v_i$ sections of $\cT_f$, where $\cT_f^i$ is the $\sO_X$-dual of $\Omega_f^i$. We locally choose a free basis of $\cT^k_f$, denoted by $(\xi_1,\dots,\xi_k)$. Then by construction, the relative log Spencer complex locally is the Koszul complex of $\sA^\bullet_{(X,D^X)}$, with actions given by multiplications of $\xi_i$, for all $i$. Since $f$ is log smooth, we know that $(\xi_1,\dots,\xi_k)$ is a regular sequence in $\sA^\bullet_{(X,D^X)}$ and the $0$-th cohomology sheaf is $f^*\sA^\bullet_{(Y,D^Y)}$. Therefore, the relative log Spencer complex is a locally free resolution of $f^*\sA^\bullet_{(Y,D^Y)}$. Using the dual-pair $(\Omega_f^i, \cT_f^i)$, we then see that 
$\cM_{f,\cL,\bullet}$ is quasi-isomorphic to 
\[\mathbf{R}^if_*(\cL^{-1}\otimes \Gr_\bullet \DR_f).\]
\end{proof}

\begin{notationsimplyfy}
To simplify notations, we write
$$\cM_{f,\cL}=\mathbf{R}^0f_*(\cL^{-1}\otimes \Gr_\bullet \DR_f),$$
for the rest of this paper. In the case that $\cL$ is trivial, we omit $\cL$ from the lower-index, that is,  we write 
$$\cM^i_f=\mathbf{R}^if_*(\Gr_\bullet \DR_f) \textup{ and } \cM_f=\cM^0_f.$$
Similarly, when $\cL$ is trivial, $\cK^0_{f,\cL,\bullet}$ (resp. $\theta^{0, k}_{f,\cL,\bullet}$) is simplified to $\cK_{f,\bullet}$ (resp. $\theta^{ k}_{f,\bullet}$) or $\cK_\bullet$ (resp. $\theta^{k}_{\bullet}$) when there is no ambiguity of $f$.
\end{notationsimplyfy}

\subsection{Direct image of filtered log $\sD$-modules and canonical extensions of variation of mixed Hodge structures}
We now discuss variations of mixed Hodge structures (VMHS) associated to log smooth morphisms. Let us refer to \cite{KasVMHS} for the definition of VMHS and admissible VMHS. In contrast to the definition in \emph{loc. cit.}, we assume the Hodge filtrations and the weight filtrations are both increasing filtrations, to be consistent with the good filtration for $\sD$-modules underlying Hodge modules. 

The following theorem is well-known; see for instance \cite{FF} and \cite{KawH}. We give it an alternative proof by using Saito's mixed Hodge modules. 
\begin{theorem}\label{thm:logHtdH}
Suppose that $f:(X,D^X)\to (Y, D^Y)$ is a projective log smooth morphism between log smooth pairs with $D^Y$ smooth (but not necessarily irreducible). We then have:
\begin{enumerate}
    \item The relative log Hodge-to-de Rham spectral sequence degenerate at $E_1$.
    \item We have an admissible VMHS $V$ underlying the local system $\mathbf{R}^if^o_*\mathbb{Q}_{U^X}$, where $f^o$ denotes the morphism $f^o: U^X\coloneqq X\setminus D^X\to U^Y\coloneqq Y\setminus D^Y$. Moreover, if $f$ is strict, then $V$ is a variation of Hodge structure (VHS).
    \item The associated graded module (with respect to the Hodge filtration) of the upper-canonical extension $V^{\ge 0}$, that is, the logarithmic extension of $V$ with eigenvalues of the monodromy of $\mathbf{R}^if^o_*\mathbb{Q}_{U^X}$ along $D^Y $ in $[0,1)$), is $\cM^i_f$.
\end{enumerate}
\end{theorem}
\begin{proof}
We first deal with the case that $D^Y$ is empty. We consider the mixed Hodge module $M$ underlying $Rj_*\Q_{U^X}[n]$, where $j:U^X\hookrightarrow X$. The underlying filtered $\sD_X$-module of $M$ is $\sO_X(*D^X)$ with 
\[\sO_X(*D^X)=\bigcup_{k\in \Z} \sO_X(kD^X)\]
and the filtration 
$$F_p\sO_X(*D^X) = \left\{
  \begin{array}{lr}
    \sO_X((p+1)D^X) ,& q \ge 0\\
    0 ,& q < 0
  \end{array}
\right.
$$
By \cite[Proposition 2.3]{BPW17}, we obtained a filtered quasi-isomorphism 
\begin{equation}\label{eq:qi111}
    f_+(\sO_X(*D^X),F_\bullet)\simeq Rf_*(\DR_f,F_\bullet DR_f)
\end{equation}
where $f_+$ denotes the filtered $\sD$-module direct image functor. By the strictness of the Hodge filtration in \cite[3.3.17]{Sa88} and \cite[2.15]{Sa90}, the filtered complex $f_+(\sO_X(*D^X),F_\bullet)$ is strict or equivalently the relative log Hodge-to-de Rham spectral sequence degenerate at $E_1$ by the filtered quasi-isomorphism \eqref{eq:qi111}. Hence Part (1) follows in this case. 

Since $R^if_*\DR_f$ is $\sO$-coherent and hence it is locally free over $\sO_Y$ as $$R^if_*\DR_f\simeq \cH^if_+\sO_X(*D^X)$$ 
are both $\sD_Y$-modules. Hence, the underlying perverse sheaf $\mathbf{R}^if^o_*\mathbb{Q}_{U^X}[n]$ is locally constant (upto a shift). By \cite[3.27]{Sa90}, we hence obtain Part (2) in this case.

If $D^Y$ is smooth but not empty, then one can use the double-strictness of direct image functor for mixed Hodge modules of normal crossing type (see \cite[\S 3]{Sa90} and also \cite{W17a}). More precisely, one applies for instance \cite[Theorem 15]{W17a} and Part (1), (2) and (3) follow in general.
\end{proof}


We also need the following weak negativity of Kodaira-Spencer kernals proved in \cite{PW}.
\begin{theorem}\cite[Theorem 4.8]{PW}\label{P:qe for -N-L}
In the situation of Theorem \ref{thm:logHtdH}, if $f$ is strict, then $\mathcal{K}^\vee_{p}$ is weakly positive for each $p\in \Z$, where $\bullet^\vee$ denotes the $\sO$-dual of $
\bullet$.  
\end{theorem}
\subsection{Adding extra boundary divisors}\label{subsect:3.4}
We assume that $f\colon (X,D^X)\to (Y,D^Y)$ is a log smooth morphism of log smooth pairs. 
If we add a divisor $D'$ to $D^Y$ (assuming $D'$ is not supported on $D^Y$) and obtain another normal crossing divisor $D'^Y$ over a big open subset of $Y$ (by getting rid of the singular locus of $D'$ and the intersection of $D'$ and $D^Y$), then we define $D'^X$ by $D'^X_h=D^X_h$ and $D'^X_v=f^{-1}D'^Y$, and denote $f':(X,D'^X)\to (Y,D'^Y)$ the new log smooth morphism (by Lemma \ref{L:log smooth out of small}). Since $f$ is smooth away from $D^Y$,  by a local computation one easily obtains that
\begin{equation}
    \Omega_f\simeq \Omega_{f'},
\end{equation}
over a big open subset of $Y$,
which further implies
\begin{equation}\label{E:id for extra boundary}
    \mathcal{M}^i_{f,\cL}\simeq \mathcal{M}^i_{f',\cL},
\end{equation}
as graded $\sA_{(Y,D^Y) }$-modules over a big open subset of $Y$. 


More generally, we have the following.
\begin{prop}\label{P:id from smooth base change}
Fix a projective log smooth morphism $f:(X, D^X)\to (Y, D^Y)$, and a morphism of log smooth pairs $\eta:(Y', D^{Y'})\to(Y,D^Y)$, with the underlying morphism of schemes $Y'\to Y$ being smooth. Let $X'=X\times_Y Y'$, $D^{Y'}$ be a SNC divisor containing $\eta^{-1}D^Y$ and $f':(X', D^{X'})\to (Y', D^{Y'})$ be the induced morphism of log smooth pairs with $D^{X'}=f'^{-1}D^{Y'} \cup \bar{\eta}^{-1}D^X$ in the following diagram:
\begin{equation*}
\begin{tikzcd}
(X', D^{X'})\arrow[r,"\bar\eta"] \arrow[d,"f'"] &(X, D^X)\arrow[d,"f"]\\
(Y', D^{Y'})\arrow[r,"\eta"] & (Y, D^Y)
\end{tikzcd}
\end{equation*}
Then, over a big open subset of $Y'$, we have 
an natural isomorphism of graded $\sA^\bullet_{(Y', D^{Y'})}$-modules
$$    \eta^*\cM_{f}\simeq \cM_{f'}.
$$
\end{prop}
\begin{proof}
We first note that, using  the identity (\ref{E:id for extra boundary}),
we only need to show the case when $D^{Y'}=\eta^{-1}D^Y$. In this case we also have $D^{X'}=\bar\eta^{-1}D^X$. The required statement then follows from Proposition \ref{P:iso M to relative DR} and the smooth base-change. 

\end{proof}

\begin{prop}\label{P:morphism due to base change}
Assume that $f:(X, D^X)\to (Y, D^Y)$ is a projective log smooth morphism with $D^X=f^{-1}D^Y$ (that is, $f$ is strict), and $\psi: Y'\to Y$ is a finite morphism branched over $D^Y$ with $\psi^{-1}D^Y$ normal crossing. Let $(X', D^{X'})$ be a log resolution of the normalization of the main component of $X\times_Y Y'$  with $D^{Y'}$ a SNC divisor containing $\psi^{-1}D^Y$ and $D^{X'}=f'^{-1}D^{Y'}$ (we can assume that $D^{X'}$ is also normal crossing after taking further log resolutions), and $f':(X', D^{X'})\to (Y', D^{Y'})$ be the induced projective morphism of log smooth pairs in the following diagram 
\begin{equation*}
\begin{tikzcd}
(X', D^{X'})\arrow[r,"\bar\psi"] \arrow[d,"f'"] &(X, D^X)\arrow[d,"f"]\\
(Y', D^{Y'})\arrow[r,"\psi"] & (Y, D^Y)
\end{tikzcd}
\end{equation*}
Then, over a big open subset of $Y'$, we have a natural inclusion of $\sA^\bullet_{(Y', D^{Y'})}$-modules
$$    \psi^*\cM_{f}\to \cM_{f'}.
$$
\end{prop}
\begin{proof}
Since over $Y'\setminus D^{Y'}$ $\psi$ is \'etale, we have $\psi^*\cM_f$ is identical with $\cM_{f'}$ over $Y'\setminus D^{Y'}$, by Proposition \ref{P:morphism due to base change}.
Then by Theorem \ref{thm:logHtdH} (3), we further have that $\cM_f$ (resp. $\cM_{f'}$) is the associated graded module of the upper-canonical extension $V^{\ge 0}_f$ (resp. $V^{\ge 0}_{f'}$ ) of the VHS of the $0$-th cohomology of the smooth fibers of $f$ (resp. $f'$). By \cite[Proposition 4.4]{PW}, we see that 
\[\psi^*V^{\ge 0}_f\subseteq V^{\ge 0}_{f'}.\] 
Then we apply \cite[Corollary 4.7]{PW} for each filtrant of  $V^{\ge 0}_f$ and $V^{\ge 0}_{f'}$ in their Hodge filtrations and obtain the required inclusion.  
\end{proof}



\subsection{Birationally invariant Hodge modules}

\begin{theorem}\label{T:common HM}
Given a projective dominant morphism $f:X\to Y$ with $X$ and $Y$ being smooth, there exists a pure Hodge module $M_f$ on $Y$, which corresponds to a generically defined VHS (by using the equivalence in \cite[3.21]{Sa90}), such that for any projective morphism $f':X'\to Y$ with $X'$ being smooth and birational to $X$, $M_f$ is isomorphic to a direct summand of $\mathcal{H}^0f'_+\Q^H_{X'}$, and its lowest filtered piece in the Hodge filtration is isomorphic to $f_{*}\omega_{f}=f'_{*}\omega_{f'}$, where $\Q^H_{X'}$ is the trivial pure Hodge module on $X'$ and $f'_+$ denotes the (derived) direct image functor for Hodge modules.  
\end{theorem}
\begin{proof}
We first fix $\tilde{X}$, a common resolution of $X$ and $X'$.
\begin{equation*}
    \begin{tikzcd}
    X\arrow[rd, "f"] &\tilde{X}\arrow[r, "\psi"]\arrow[d, "\tilde{f}"] \arrow[l,"\phi"]& X'\arrow[ld, "f'"]\\
     & Y 
    \end{tikzcd}
\end{equation*}
Note that $f'_{*}\omega_{f'}=f_{*}\omega_{f}=\tilde{f}_*\omega_{\tilde{f}}$ as $X$ and $X'$ are both smooth. 
Now we define $M_f$ the smallest sub-Hodge module of $\mathcal{H}^0\tilde f_+\Q^H_{\tilde X}$
that contains $\tilde{f}_*\omega_{\tilde{f}}$.
More precisely, due to the semi-simplicity of pure Hodge modules (see \cite[\S 5]{Sa88}), we define $\cM$ the sub-pure Hodge module of $\mathcal{H}^0\tilde f_+\Q^H_{\tilde X}$  that consists of all the simple sub-pure Hodge modules that intersect $\tilde{f}_*\omega_{\tilde{f}}$ non-trivially. Thanks to the decomposition theorem of the direct image of pure Hodge modules under projective morphism \cite[Th\'eor\`eme 1 and Corollaire 3]{Sa88} (see also \cite[\S16]{Sc14}) and the equivalence in \cite[3.21]{Sa90}, we have that both $\mathcal{H}^0 f_+\Q^H_{X}$ and $\mathcal{H}^0f'_+\Q^H_{X'}$ are direct summands of $\mathcal{H}^0\tilde f_+\Q^H_{\tilde X}$, and all three pure Hodge modules share the same lowest filtered piece. As a consequence, we conclude that $M_f$ satisfies the conditions.
\end{proof}

\begin{prop}\label{P:bir invar of M}
Fix two projective morphisms of log smooth pairs
$$f_i:(X_i, D^{X_i})\to (Y, D^Y), \text{ } i=1, 2,$$
with $D^Y$ a smooth divisor (but not necessarily irreducible) and $X_1$ and $X_2$ birational. Assume that $D^{X_i}=f^{-1}_iD^Y$ (in particular, $D^{X_i}$ are SNC divisors) and $f_i$ are smooth over $Y\setminus D^Y$, for $i=1,2$. Recalling the notations in \S \ref{SS:Direct image of graded log D-modules}, we can find a graded $\sA^\bullet_{(Y, D^Y)}$-module $\hat{\cM}$, which is a direct summand of both $\cM_{f_i}$ for $i=1,2$. Furthermore, all of $\hat{\cM}, \cM_{f_i}$ share the same lowest graded piece.
\end{prop}
\begin{proof}
Recall that $\cM_{f_i}$ are isomorphic to the associated graded modules of the upper-canonical extension along $D^Y$ of the VHS given by the restriciton of $\mathcal{H}^0f_{i+}\Q^H_{X_i}$ on $Y\setminus D^Y$
(see Theorem \ref{thm:logHtdH}(2) and (3)).
By the previous theorem, we get a pure Hodge module $M_{f_1}$ that corresponds to a VHS on $Y\setminus D^Y$, called $V$. 
Now we set $\hat{\cM}$ to be the associated graded module of the canonical extension along $D^Y$ of $V$, with the real part of the eigenvalues of the residues in $[0,1)$. 
\end{proof}

\subsection{Construction of Higgs sheaves from a section}\label{SS:Sub higgs sheaves from a section}
Suppose that we have the following diagram of projective morphisms of log smooth pairs: $$
\begin{tikzcd}
(Z,D^Z)\arrow[r,"\pi"]\arrow[dr,"g"] &(X,D^X)\arrow[d,"f"]\\  &(Y,D^Y),
\end{tikzcd}
$$
with $\pi$ generically finite, and both $f$ and $g$ being log smooth. By forgetting the horizontal part of $D^Z$ respect to $g$, we denote $g^0:(Z, D^Z_v)\to (Y, D^Y)$. Note that, by Lemma \ref{L:log smooth equivalence} and Lemma \ref{L:log smooth out of small}, $g^0$ is log smooth at least over a big open subset of $Y$. For simplicity, in the rest of this subsection, we assume $g^0$ is log smooth by removal a small closed subset of $Y$.

We have a natural morphism $\pi^*\Omega_X(\log D^X)\to \Omega_Z(\log D^Z),$ which induces a natural graded morphism of graded complexes
\begin{equation}\label{E:mor of rel DR induced by mor}
    \pi^*\Gr_\bullet \DR_f\to \Gr_\bullet \DR_{g}.
\end{equation}

For invertible sheaves $\cL$ and $\cA$ on $X$ and $Y$ respectively, we assume the following two assumptions:\\
(Assumption. 1)  $H^0(Z, \pi^*\mathcal{L}\otimes g^*\cA^{-1})\neq 0.$\\
(Assumption. 2) the support of the effective Cartier divisor $(\theta)$ contains $D^Z_h$.

By (Assumption. 1), fixing one non-trivial section 
$$\theta\in H^0(Z, \pi^*\mathcal{L}\otimes g^*\cA^{-1})\neq 0,$$  
we have an induced inclusion
$$\pi^*\mathcal{L}^{-1} \to g^*\cA^{-1}.$$
We then have an induced graded morphism 
\begin{equation}\label{E:mor bet relative log from cyclic}
    \pi^*(\Gr_\bullet \DR_f\otimes \mathcal{L}^{-1})\to \Gr_\bullet \DR_{g} \otimes g^*\cA^{-1}.
\end{equation}
as the composition of the natural morphisms
\[\pi^*(\Gr_\bullet \DR_f\otimes_\sO \mathcal{L}^{-1})\to\Gr_\bullet \DR_{g} \otimes_\sO \pi^*\cL^{-1}\to \Gr_\bullet \DR_{g} \otimes_\sO g^*\cA^{-1}.\]

By (Assumption. 2), the morphism (\ref{E:mor bet relative log from cyclic}) naturally factors through 
$$\pi^*(\Gr_\bullet \DR_f\otimes \mathcal{L}^{-1})\to \Gr_\bullet \DR_{g}(-D^Z_h) \otimes g^*(\cA^{-1}).$$
Note that the natural inclusion $\Omega_Z(\log D^Z)(-D^Z_h)\to \Omega_Z(\log D^Z_v)$ induces a natural inclusion of complexes:
$$\Gr_\bullet \DR_{g}(-D^Z_h)\to \Gr_\bullet \DR_{g^0}.$$
We then have an induced morphism 
$$\pi^*(\Gr_\bullet \DR_f\otimes \mathcal{L}^{-1})\to \Gr_\bullet \DR_{g^0} \otimes g^*\cA^{-1}.$$
This further induces a morphism 
$$\eta_\theta:\mathcal{M}_{f,\mathcal{L}}\to \mathcal{M}_{g^0, g^*\cA},$$ 
as graded $\sA^\bullet_{(Y, D^Y)}$-modules by Proposition \ref{P:iso M to relative DR}. Furthermore, the lowest graded piece of $\eta_\theta$:
$$g_*\pi^*(\omega_f\otimes \cL^{-1})\to g_*\omega_{g^0}\otimes \cA^{-1},$$
is induced by the morphism
$\pi^*(\omega_f\otimes \cL^{-1})\to \omega_{g^0}\otimes g^*\cA^{-1}$. Since $\pi^*(\omega_f\otimes \cL^{-1})\to \omega_{g^0}\otimes g^*\cA^{-1}$ is induced by $\theta$ and hence injective, by the left exactness of $g_*$, we see that 
$$g_*\pi^*(\omega_f\otimes \cL^{-1})\to g_*\omega_{g^0}\otimes \cA^{-1},$$
is also injective.

\section{Geometric construction}\label{S:Geo const}
\subsection{Stable reduction for families with arbitrary variation}
\label{SS:Ass to max}
In \cite[\S 4]{WW19}, we introduced stable reduction for log smooth families with maximal variation. In this section, we discuss stable reduction for log smooth families with arbitrary variation.

Suppose that $f_V\colon (U,D^U) \to V$ is a log smooth family of projective log smooth pairs of log general type, with $(U,D^U)$ being klt and $V$ smooth. Consider the relative canonical model 
\[f_{V,c}\colon (U_c, D^{U_c})\to V.\]
We write $v$ the volume of $K_{U_{c,y}}+D^{U_c}_y$ for $y\in V$ (by invariance of pluri-genera, $v$ is constant over $V$). Let $I$ be a finite coefficient set $I$ closed under addition and containing the coefficients of $D^{U_c}$. We consider the coarse moduli space, denoted by $\fM$, for stable log varieties of a fixed dimension, volume $v$ and the coefficient set $I$. Let us refer to \cite[\S 6]{KP16} for the construction of $\fM$ and related properties. It is proved in \emph{loc. cit.} that $\fM$ is a projective reduced scheme and the corresponding moduli stack is an Deligne-Mumford stack. The relative canonical model $f_{V,c}$ (, which is a stable family thanks to invariance of pluri-genera in the log smooth case \cite[Theorem 4.2]{HMX18},) induces a moduli map $V\rarr \fM.$
Since $\fM$ is projective, we take a projective compactification $Y$ of $V$ so that $E^Y=Y\setminus V$ is normal crossing and the moduli map extends to a projective map $Y\to\fM$ (we probably need to replace $f_V$ by a birational model).  
\begin{prop}\label{P:stable reduction}
Under the above construction, replacing $Y$ by a further resolution and removing a small closed subset, we can construct the following commutative diagram
\begin{equation}\label{E:reddiagram}
\begin{tikzcd}
(X_c,D^{X_c}) \arrow[d, "f_c"]\arrow[drr, phantom, "\square_1"] & (X'_{1,c}, D^{X'_{1,c}})\arrow[rd, "f'_{1,c}"]\arrow[l, swap, "\psi_c"]\arrow{rr}{\rho}& &(X'_{2,c}, D^{X'_{2,c}})\arrow[dr, phantom, "\square_2"] \arrow[ld, "f'_{2,c}"]\arrow[r,"\eta_c"] & (X^\sharp_c, D^{X^\sharp_c})\arrow[d, "f^\sharp_c"]  \\
Y & & Y'\arrow[ll, "\psi"] \arrow[rr, "\eta"] & & Y^\sharp 
\end{tikzcd}
\end{equation}
such that
\begin{enumerate}
\setcounter{enumi}{-1}
\item $(X^\bullet_c,D^{X_c^\bullet})$ are all klt pairs with $\bullet= \emptyset, '_1, '_2, ^\sharp$;
   \item $f^\sharp_c$ is a stable family with maximal variation;
   \item the square $\square_1$ is Cartesian over $V$;
   \item the square $\square_2$  is Cartesian;
   \item $\eta$ is a dominant smooth morphism, not necessarily proper;
   \item $\psi$ is finite and flat; 
   \item  for all sufficiently positive and divisible $m$, we have isomorphisms
   \begin{align*}
          f'_{2,c*}\omega^m_{X'_{2,c}/Y'}(mD^{X'_{2,c}}) &\simeq \eta^* f^\sharp_*\omega^m_{X^\sharp_c/Y^\sharp}(mD^{X^\sharp_c});\\
       \det f'_{2,c*}\omega^m_{X'_{2,c}/Y'}(mD^{X'_{2,c}}) &\simeq \eta^*\det f^\sharp_*\omega^m_{X^\sharp_c/Y^\sharp}(mD^{X^\sharp_c});
   \end{align*}
   \item for sufficiently positive and divisible $N_m$ (depending on $m$), we can find a line bundle $\cA_m$ on $Y$ such that 
   $$\psi^*\cA_m \simeq (\det f'_{2,c*}\omega^m_{X'_{2,c}/Y'}(mD^{X'_{2,c}}))^{N_m}.$$
\end{enumerate}
\end{prop}
\begin{proof}
By \cite[Theorem 16.6]{LMB}, $\fM$ has a finite covering $S\to \fM$ which is induced by a stable family over $S$ in the moduli stack of $\fM$. We then take $Y'$ to be a desingularization of the main component of $Y\times_\fM S$ (the component dominant $Y'$). Then we got a generic finite map $\psi\colon Y'\to Y$. 
We then take the Stein factorization of the induced morphism $Y'\to S$ and obtain a fibration $\eta: Y'\to Y^\sharp$. Thanks to Raynaud-Gruson flattening theorem (\cite[Th\'eor\`em 5.2.2]{RG71}), after replacing $Y^\sharp$ by a further resolution, we can assume that $\eta\colon Y'\to Y^\sharp$ is a flat fibration, with $Y^\sharp$ being projective and smooth. Note that $Y'$ might have more than 1 components, and we then replace $Y'$ by the main component.
Now we take a Kawamata covering ${Y^\sharp}'\to Y^\sharp$ so that there exists a $Y''$, a desingularization of the main component of $Y'\times_{Y^\sharp}{Y^\sharp}'$ satisfying that $Y''\to {Y^\sharp}'$ is semistable in codimension 1; semistable means that the morphism is flat with reduced fibres, see for instance \cite{AK00} for a stronger result. In particular, we can remove a small closed subset of $Y''$ to make the morphism be smooth. Replace $Y'$ by $Y''$ and $Y^\sharp$ by ${Y^\sharp}'$. To finish the construction of the bases, we remove the largest reduced divisor $F$ on $Y'$ such that $\codim \psi(F)\ge 2$, and replace $Y$ by the image of $Y'$ (with $F$ removal) under $\psi$. Note that we only removed a small subset of $Y$.

Now we start to construct the $X^\bullet$ level. We first take a klt log smooth pair $(X,D^X)$ with a projective morphism $f_Y:(X,D^X)\to Y$ so that $f_Y|_V=f_V$. Since general fibers of $f$ are of log general type, by \cite[Theorem 1.2]{BCHM} we take $f_c$ to be the relative log canonical model of $f_Y$. In particular, $(X,D^{X_c})$ is a klt pair. We then set $X'_{1,c}$, the main component of the normalization of $X_c\times_Y Y'$, with boundary divisor $D^{X'_{1,c}}$ given by $\omega_{X'_{1,c}}(D^{X'_{1,c}})= \psi_c^* \omega_{X_c}(D^{X_c})$. 
Since $\psi$ is finite, so is $\psi_c$, hence $(X'_{1,c}, D^{X'_{1,c}})$ is klt by \cite[Proposition 5.20]{KM98}. The induced morphism $Y^\sharp \to S$ is generic finite, so it induces a stable family $f_c^\sharp:(X^\sharp_c, D^{X^\sharp_c})\to Y^\sharp$. $f'_{2,c}: (X'_{2,c}, D^{X'_{2,c}})\to Y'$ is defined as the stable family induced by $f^\sharp_c$ under the smooth base change $\eta$. 
Note that by construction, we have that $f'_{1,c}$ and $f'_{2,c}$ coincide on $\psi^{-1} V$, hence we can blowup $(X'_{1,c}, D^{X'_{1,c}})$ to make $\rho$ a morphism, without changing $f'_{1,c}$ over $\psi^{-1} V$. Now the morphisms satisfy (1) to (5), and (6) follows by the flat base change. 

To prove (7), we use the argument similar to the proof of \cite[Corollary ix)]{VZ03}.
Note that, over an open set $V_0\subset Y$, since $f'_{2,c}$ can be induced by $f_c$ using a flat base change $\psi$, we can canonically identifying $\psi^*\det f_*\omega^m_{X_{c}/Y}$ and $\det f'_{2,c*}\omega^m_{X'_{2,c}/Y'}(mD^{X'_{2,c}})$ over $V'_0:=\psi^{-1}V_0$, and set $B$ the divisor support on $Y'\setminus V'_0$, satisfying
$$\psi^*\det f_*\omega^m_{X_{c}/Y}(mD^{X_c})= \det f'_{2,c*}\omega^m_{X'_{2,c}/Y'}(mD^{X'_{2,c}})(B),$$
that is also canonically defined. Now we only need to show that $B$ is the pullback of some $\Q$-divisor on $Y$, so only need to show that, for any two components $B_1$ and $B_2$, if they have a same image under $\psi$, then they share a same coefficient. Now we only need to show that over a general point $p$ of $\psi(B_1)$. Take a local curve $Q$ that only intersect with $\psi(B)$ at $p$. Now $f_{Q}$, the restriction of $f_c$ onto $Q$, is stable over $Q\setminus p$, so we can find a cyclic cover $\pi: \tilde Q\to Q$ that is totally ramified at $p$, such that the induced new family $\tilde f_{Q}$ over $\tilde Q$ is stable and compatible with the base change of $f_Q$ over $Q\setminus p$. Denote $Q'$ any irreducible component of $\psi^{-1} Q$, and let $\tilde Q'=\tilde Q\times_Q Q'$. Although we can induce two stable families over $\tilde Q'$, but since they are the same over general point, hence they are the same family, by the properness of the moduli functor, and we denote the family by $f_{\tilde Q'}$. Note that $f_{\tilde Q}, f_{Q'}$ and $f_{\tilde Q'}$ are all compatible with flat base change, so we only need to verify the proposition on $f_Q$ and $f_{\tilde Q}.$ This is true, since $\pi$ is totally ramified at $p$.
\end{proof}

Since $f^\sharp_c$ is a stable family with maximal variation, we have $\det f^\sharp_{c*}\omega^m_{X_c^\sharp/Y^\sharp}(mD^{X_c^\sharp})$ is big by \cite[Theorem 7.1]{KP16}.
We now fix a pair of $m$ and $N_m$ as in the previous proposition. To simplify notation, we set 
\begin{align*}
    \cA^\sharp&:= (\det f^\sharp_{c*}\omega^m_{X_c^\sharp/Y^\sharp}(mD^{X_c^\sharp}))^{N_m}\\
    \cA'&:= (\det f'_{2,c*}\omega^m_{X'_{2,c}/Y'}(mD^{X'_{2,c}}))^{N_m}\\
    \cA&:=\cA_m.
\end{align*}
By Part (6) and (7) in Proposition \ref{P:stable reduction}, we have
\begin{equation}\label{eq:3lbdle}
    \psi^*\cA=\cA'=\eta^*\cA^\sharp \textup{ and }\kappa (\cA) = \kappa (\cA')= \kappa(\cA^\sharp)=\Var(f_V).
\end{equation}


\subsection{Birational refinement of the stable reduction}\label{SS:Further log resolutions}
In this subsection, we refine the stable reduction in Proposition \ref{P:stable reduction}, which will be used to construct Viehweg-Zuo sheaf in later section. 
We first show the following lemma.

\begin{lemma}
Assume that we have two log smooth projective morphisms 
$$f_i: (X_i, D^{X_i})\to (Y, D^Y), \text{with } i=1,2,$$
whose generic fibers share the same birational equivalent model, i.e. we have log resolutions over $\xi$, $\phi_{i,\xi}: (X_\xi, D^{X_\xi})\to (X_{i,\xi}, D^{X_{i,\xi}}),$ satisfying $$\phi^*_{i,\xi}\omega_{X_\xi}(D^{X_{i,\xi}})\simeq \omega_{X_\xi}(D^{X_\xi}),$$ 
where $\xi$ is the generic point of $Y$. After adding components to $D^Y$ if necessary (cf. \S \ref{subsect:3.4}), forgetting the information over a small closed subset of $Y$, and replacing $D^{X_i}_v$ by $f^{-1}(D^Y)$, we have $f_{1*}\omega_{f_1}^m= f_{2*}\omega_{f_2}^m$.
\end{lemma}
\begin{proof}
According to the assumption, we can have a common log resolution, hence the following commutative diagram, with all pairs are log smooth, 
\begin{equation*}
\begin{tikzcd}
(X_1, D^{X_1})\arrow[rd, swap, "f_1"]& (X, D^{X})\arrow[l, "\phi_1"]\arrow[r,"\phi_2"]\arrow[d, "f"] &(X_2, D^{X_2})\arrow[ld, "f_2"]
\\
 & (Y, D^{Y}),
\end{tikzcd}
\end{equation*}
so that $D^X|_{X_\xi}=D^{X_\xi}$. We can assume that every irreducible component of $D^X$ is dominant over $Y$, i.e. the vertical part of $D^X$ over $Y$ is zero. 
After we extend $D^Y$ and forget a small closed subset of $Y$, we can assume that $f$ is smooth away from $D^Y$ without changing log smoothness of $f_1$ and $f_2$. We then replace the vertical parts of $D^{X_1}$, $D^X$ and $D^{X_2}$ by $f_1^{-1}(D^Y)$, $f^{-1}(D^Y)$ and $f_2^{-1}(D^Y)$ respectively.
Using Lemma \ref{L:log smooth out of small}, we can assume that all vertical morphisms are projective log smooth morphisms.

To prove the statement, by symmetry we only need to show that $\phi_{1*}(\omega_X(D^X))^m=(\omega_{X_1}(D^{X_1}))^m$. Consider a Cartier divisor $B$ with
$$\cO(B)=(\omega_X(D^X))^{m}\otimes \phi_1^*(\omega_{X_1}(D^{X_1}))^{-m},$$
which is an $\phi_1$-exceptional divisor. Since $f$ is log smooth, any irreducible $\phi_1$-exceptional divisor $E$ on $X$ maps to $D^Y$ or dominates $Y$. By the assumption of $f_i$ over the generic point of $Y$, $E$ cannot dominant $Y$.
Then $E$ maps to $D^Y$ and $E$ is $f$-exceptional. Hence the coefficient of $E$ in $D^X$ is $1$ by construction (see (Assumption. 0) in \S \ref{subsec:3.1}). Since $(X_1, D^{X_1})$ is log canonical (by (Assumption. 0) in \S \ref{subsec:3.1}), the coefficient of $E$ in $B$ is positive, which concludes the proof.
\end{proof}

\begin{prop}\label{P:stable reduction log smooth}
Assume that we have a commutative diagram (\ref{E:reddiagram}) as in Proposition \ref{P:stable reduction}, we can always find a log resolution $\pi:(X, D^X)\to (X_c, D^{X_c})$, so that we can have the following commutative diagram, allowing forgetting a small closed subset of $Y$, $Y'$ and $Y^\sharp$.
\begin{equation}\label{E:reddiagram log}
\begin{tikzcd}
(X,D^X) \arrow[d, "f"]& (X'_1, D^{X'_1})\arrow[dl, phantom, "\square_1"] \arrow[rd, swap, "f'_1"]\arrow[l, "\bar\psi"]& (X', D^{X'})\arrow[l, "\phi_1"]\arrow[r,"\phi_2"]\arrow[d, "f'"] &(X'_2, D^{X'_2})\arrow[dr, phantom, "\square_2"] \arrow[ld, "f'_2"]\arrow[r,"\bar\eta"] & (X^\sharp, D^{X^\sharp})\arrow[d, "f^\sharp"]
\\
(Y, D^Y) & & (Y', D^{Y'})\arrow[ll, "\psi"] \arrow[rr, "\eta"] & & (Y^\sharp, D^{Y^\sharp })
\end{tikzcd}
\end{equation}

such that
\begin{enumerate}
   \item all the log pairs are log smooth;
   \item $f$ is a projective compactification of $f_V$;
   \item all the downwards morphisms are log smooth, with vertical boundaries are reduced and the coefficients of horizontal boundaries are in $(0, 1)$;
   \item the square $\square_1$ is Cartesian over $Y\setminus D^Y$ and $\psi$ is \'etale over $Y\setminus D^Y$;
   \item the square $\square_2$  is Cartesian;
   \item all $\phi_\bullet$ are birational, with $\bullet=1,2$;
   \item $\bar\psi$ factors through the natural morphism $\psi_0:X'_0\to X$, where $X'_0$ is the normalization of the main component of $X\times_Y Y',$ and $\phi_0:X'_1\to X'_0$ is birational, and denote $f'_0:X'_0\to Y'$ the induced morphism; 
   \item for all $m$ sufficiently positive and divisible, we have isomorphisms 
    \begin{equation*}
            f'_{0*}\psi_0^*\omega^m_f\simeq f'_{1*}\bar\psi^*\omega^m_f\simeq f'_{1*}\omega^m_{f'_1} \simeq f'_*\omega^m_{f'}\simeq f'_{2*}\omega^m_{f'_2}\simeq  \eta^*f^\sharp_* \omega^m_{ f^\sharp};
    \end{equation*}
   \item denoting $\lceil f^\bullet\rceil:(X^\bullet, \lceil D^{X^\bullet}\rceil)\to (Y^\bullet, D^{Y^\bullet})$, with $\bullet= \emptyset, '_1, '_2, ^\sharp$, i.e. only considering the boundary with reduced structure, for all $m$ sufficiently positive and divisible, we have isomorphisms 
       \begin{equation*}
            f'_{0*}\psi_0^*\omega^m_{\lceil f\rceil}\simeq f'_{1*}\bar\psi^*\omega^m_{\lceil f\rceil}\simeq f'_{1*}\omega^m_{\lceil f'_1\rceil} \simeq f'_*\omega^m_{\lceil f'\rceil}\simeq f'_{2*}\omega^m_{\lceil f'_2\rceil}\simeq \eta^*f^\sharp_* \omega^m_{\lceil f^\sharp\rceil}.
    \end{equation*}
\end{enumerate}

\end{prop}
\begin{proof}
From Proposition \ref{P:stable reduction}, we first replace $(X^\bullet_c, D^{X^\bullet_c})$ by their log  resolutions $(X^\bullet, D^{X^\bullet})$, with $\bullet= \emptyset, '_1, '_2, ^\sharp$, 
and the horizontal part of the boundary divisor $D^{X^\bullet}$ is defined by taking the positive and horizontal part of $\omega^{-1}_{X^\bullet}\otimes \pi^{\bullet *}\omega_{X^\bullet_c}(D^{X^\bullet_c})$, where $ \pi^{\bullet}$ are the corresponding log resolutions. Since $(X^\bullet_c, D^{X^\bullet_c})$ are klt, so are $(X^\bullet, D^{X^\bullet})$. Moreover, by the construction of $f_c$ in Proposition \ref{P:stable reduction}, we can assume $f:(X,D^X)\to Y$ is a projective compactification of the initial log smooth family $f_V:(V,D^U)\to V$. In particular, we can keep track the initial boundary divisor $E^Y=Y\setminus V$ because ultimately we are interested in the log Kodaira dimension of $V$. 

We do not consider the vertical parts of $D^{X^\bullet}$ at this stage, but they will be specified after we fixed the boundary divisors on $Y^\bullet$. Since base changing by an étale or smooth morphism will keep the log smoothness, we can keep $\square_1$ and $\square_2$ are Cartesian over the assigned locus, and make $\bar\psi$ and $\bar\eta$ are morphisms of log smooth pairs. By the construction of $f'_{1,c}$ in Proposition \ref{P:stable reduction}, we have (7).

Set $D^Y$ and $D^{Y^\sharp}$ be divisors on $Y$ and $Y^\sharp$ respectively, so that $f$ and $f^\sharp$ are log smooth over the base outside of the boundary (cf. Lemma \ref{L:log smooth equivalence}). 
By expanding $D^Y$, we can assume that $\psi$ is \'etale over $Y\setminus D^Y$.
Set $D^{Y'}$ a divisor on $Y'$, so that it contains both $\psi^{-1} D^Y$ and $\eta^{-1} D^{Y^\sharp}$, and both $f'_1$ and $f'_2$ are log smooth over $Y\setminus D^{Y'}$. By Lemma \ref{L:log smooth out of small}, also Remark \ref{R:codim 2}, we can assume that $f'_1$ and $f'_2$ are log smooth, by removing a small subset of $Y'$. Hence, we can apply the previous lemma, by further expending $D^{Y'}$, constructing $f'$ as in the proof of the lemma, and replacing all vertical boundaries by the inverse image $D^{Y'}$, we have the identities of the third term to the fifth term of (8). Then, we replace $D^Y$ and $D^{Y^\sharp}$, by $\psi_* D^{Y'}$ and $\eta_* D^{Y'}.$ Now we expanding $D^{X^\bullet}$ by setting $D^{X^\bullet_v}=f^{\bullet-1}D^{Y^\bullet}$. Meanwhile we take further log resolutions of $(X^\bullet, D^{X^\bullet})$ if needed, to keep them being log smooth, but do not change the locus that are already log smooth, so that we can keep the Cartesian condition. At last, remove some small subsets of $Y$, $Y'$, and $Y^\sharp$, we have all the downwards morphisms are log smooth. 

\begin{equation*}
\begin{tikzcd}
(X,D^X) \arrow[d, "f"]& (X'_0, D^{X'_0})\arrow[rd, swap, "f'_0"]\arrow[l, "\bar\psi_0"]& (X'_1, D^{X'_1})\arrow[l, "\phi_0"]\arrow[d, "f'_1"]\arrow[ll, bend right=20, swap, "\bar\psi"]
\\
(Y, D^Y) & & (Y', D^{Y'})\arrow[ll, "\psi"].
\end{tikzcd}
\end{equation*}
The first identity in (8) is due to the projection formula. Note that $\psi$ is \'etale  outside of the boundary, so we have $\psi^*\omega_Y(D^Y)\simeq \omega_{Y'}(D^{Y'})$. Set $D^{X'_0}$ by $\bar\psi_0^*\omega_X(D^X)=\omega_{X'_0}(D^{X'_0})$. Since $\square_1$ is Cartesian over $Y\setminus D^Y$, so we have that the ramification locus of $\bar\psi_0$ is contained in $D^X_v$, which is reduced. This implies that $D^{X'_0}_v$ is also reduced and $\phi_{0*}(D^{X'_1})=D^{X'_0}$.  Hence the support of $(\bar\psi^{*}\omega_X(D^X))^{-1}\otimes \omega_{X'_1}(D^{X'_1}) $ is contained in the exceptional divisor respect to $\phi_0$, hence also contained in $D^{X'_1}_v$, and it is an effective $\Q$-divisor, due to $(X, D^X)$ is log canonical, and $D^{X'_1}_v$ is reduced. This implies that $f'_{1*}\bar\psi^{*}\omega^m_X(mD^X)\simeq f'_{1*}\omega^m_{X'_1}(mD^{X'_1})$, which concludes the second identity. 

The last identity of (8) can be deduced by flat base change, combining the fact that we have $\eta$ is smooth and $X'_2=X^\sharp \times_{Y^\sharp} Y'$, so $\bar\eta^* D^{X^\sharp}= D^{X'_2}$ and $\eta^*D^{Y^\sharp}=D^{Y'}.$

Since we only need that all log pairs are log canonical in proving all identities in (8), we get (9).
\end{proof}

\subsection{Viehweg's fiber product trick and construction of cyclic covering}
\label{SS:Product trick and cyclic covering}
We continue assuming that we are in the situation of Proposition \ref{P:stable reduction}.
We then have three line bundles $\cA^\bullet$ on $Y^\bullet$ respectively, with $\bullet= \emptyset, '$ or $\sharp$, satisfying 
\begin{equation}\label{E: A's relation}
    \psi^*\cA= \cA'=\eta^*\cA^\sharp
\end{equation}
as in \eqref{eq:3lbdle}

Recall that  $\cA^\sharp= (\det f^\sharp_{c*}\omega^m_{X_c^\sharp/Y^\sharp}(mD^{X^\sharp}))^{N_m}$, and $\cA^{\sharp}$ is a big line bundle. Hence, we can assume that $\cA^{\sharp N} (-D^{Y^\sharp})$ has a global section, for some positive integer $N$, where $D^{Y^\sharp}$ is the divisor as in \eqref{E:reddiagram log}. Since $f_c$ is stable, we see that $f^\sharp_{c*}\omega_{X^\sharp_c/Y^\sharp}^m(mD^{X^\sharp_c})$ is reflexive. 
Then, we have a natural inclusion 
$$\det f^\sharp_{c*}\omega_{X^\sharp_c/Y^\sharp}^m(mD^{X^\sharp_c}) \to (f^\sharp_{c*}\omega_{X^\sharp_c/Y^\sharp}^m(mD^{X^\sharp_c}))^{\otimes r_0},$$
where $r_0$ is the rank of $f^\sharp_{c*}\omega_{X^\sharp_c/Y^\sharp}^m(mD^{X^\sharp_c})$.
These give us a sequence of inclusions
$$(\cA^{\sharp}(D^{Y^\sharp}))^m\to \cA^{\sharp m(N+1)} \to (f^\sharp_{c*}\omega_{X^\sharp_c/Y^\sharp}^m(mD^{X^\sharp_c}))^{\otimes r},$$
where $r=m\cdot N_m(N+1).$
In particular, we have 
$$H^0(Y^\sharp, (f^\sharp_{c*}\omega_{X^\sharp_c/Y^\sharp}^m(mD^{X^\sharp_c}))^{\otimes r} \otimes (\cA^{\sharp}(D^{Y^\sharp}))^{-m})\neq 0.$$ 

We recall Viehweg's fiber product trick (see \cite{Vie83} and also \cite[Section 4]{WW19} for the log case). Suppose that we have a morphism $f:X\to Y$ with $X$ and $Y$ both smooth. We then define 
\[X^r_Y:=X\times_YX\times_Y\cdots\times_YX,\]
the $r$-th fiber product, and 
\[f^r:X^r_Y\to Y\]
the induced morphism. If $D^X$ is a $\Q$-Cartier divisor on $X$, we write
\[(D^{X})^r_Y=\sum_{i=1}^rp_i^*D^X,\]
where $p_i\colon X^r_Y\rarr X$ is the $i$-th projection; $X^r_Y$ and $(D^{X})^r_Y$ are also denoted by $X^r$ and $D^r$ respectively, if $f$ is obvious from the context. 

We first apply the fiber product trick on $f^\sharp_c$, and have 
$$(f^{\sharp}_c)^r:((X^{\sharp}_c)^r, (D^{X^{\sharp}_c})^r)\to Y^\sharp.$$ 
By \cite[Lemma 4.2]{WW19}, we have that $(f^{\sharp}_c)^r$ is also a stable family. Then by the flat base change, we have $$(f^\sharp_{c*}\omega_{X^\sharp_c/Y^\sharp}^m(mD^{X^\sharp_c}))^r \otimes (\cA^{\sharp}(D^{Y^\sharp}))^{-m}\simeq (f^\sharp_{c})^r_*\omega^m_{(X^{\sharp}_c)^r/Y^\sharp}(m(D^{X^{\sharp}_c})^r)\otimes (\cA^{\sharp}(D^{Y^\sharp}))^{-m}.$$
Therefore, 
$$H^0(Y^\sharp, (f^\sharp_{c})^r_*\omega^m_{(X^{\sharp}_c)^r/Y^\sharp}(m(D^{X^{\sharp}_c})^r)\otimes (\cA^{\sharp}(D^{Y^\sharp}))^{-m})\neq 0.$$ 
Note that, by the construction of (\ref{E:reddiagram log}), $D^{Y^\sharp}$ does not depend on $r$. 

Now we apply the fiber product trick on the log smooth morphism $f^\sharp$ in (\ref{E:reddiagram log}) and obtain 
\[(f^\sharp)^r: ((X^\sharp)^r, (D^{X^\sharp})^r)\to Y^\sharp.\]
We then replace $(X^\sharp)^r$ by a log resolution of the main component of the normalization of $(X^{\sharp})^r$ and $(f^\sharp)^r$ by the induced morphism. Note that these operations do not change $(f^\sharp)^r$ over $Y^\sharp\setminus D^{Y^\sharp}.$ Hence, we replace $(D^{X^\sharp})^r$ by the closure of $(D^{X^\sharp})^r|_{((f^\sharp)^r)^{-1}(Y^\sharp\setminus D^\sharp)}$ inside the new $(X^\sharp)^r$. We can assume that $(D^{X^\sharp})^r$ is normal crossing by taking further log resolutions. We then add  $((f^\sharp)^r)^{-1}(D^{Y^\sharp})$ to  $(D^{X^\sharp})^r$ and obtain a morphism of log smooth pairs
\[(f^\sharp)^r: ((X^\sharp)^r, (D^{X^\sharp})^r)\to (Y^\sharp,D^{Y^\sharp}).\]
By Lemma \ref{L:log smooth out of small}, $(f^\sharp)^r$ is log smooth after getting rid of a small closed subset of $Y^\sharp$. 

In the diagram \eqref{E:reddiagram log}, we replace $f^\sharp$ by the new induced morphism $(f^\sharp)^r:(X^\sharp, D^{X^\sharp})\to (Y, D^{Y^\sharp}),$ and replace $f^\sharp_c$ by $(f^{\sharp}_c)^r$. Then, by a local computation, we have that
$$f^\sharp_*\omega^m_{f^\sharp}\otimes \cA^{\sharp-m} \supset f^\sharp_{c*}\omega^m_{X^\sharp_c/Y^\sharp}(mD^{X^\sharp_c})\otimes (\cA^{\sharp}(D^{Y^\sharp}))^{-m}$$
In particular, $f^\sharp_*\omega^m_{f^\sharp}\otimes \cA^{\sharp-m}$
has non-trivial global sections. 

We similarly apply the fiber product trick on all vertical morphisms in (\ref{E:reddiagram log}), and then, replace all $(X^\bullet,D^{X^\bullet})$ by $(X^{\bullet\times r}_{Y^\bullet}, D^{X^\bullet\times r}_{Y^\bullet})$, as we did in the $\bullet= ^\sharp$ case. We end up with a new commutative diagram as (\ref{E:reddiagram log}), satisfying all properties in Proposition \ref{P:stable reduction log smooth}. Hence, we have that for any sufficient positive and divisible $m$,
$$
(f'_{0*}\psi_0^*(\omega^m_{f}\otimes f^*\cA^{-m}))
\simeq (f'_{1*}\omega^m_{f'_1}\otimes \cA'^{-m})
\simeq (f'_{2*}\omega^m_{f'_2}\otimes \cA'^{-m})
\simeq \eta^*(f^\sharp_*\omega^m_{f^\sharp}\otimes \cA^{\sharp-m}),
$$
and they all have non-trivial global sections.

Since $\psi_0$ is finite, we can find $l>0,$ such that $(\omega_f\otimes f^* \cA^{-1})^{ml}$, has non-trivial global sections, thanks to Lemma \ref{lm:decffinite}. From now on, we replace $ml$ by $m$, and denote $\cA'_1:=\cA'_2:=\cA'$ on $Y'$, hence we can find four sections
$$    \bar s^\bullet \in H^0(X^\bullet, (\omega_{f^\bullet}\otimes f^{\bullet*}(\cA^\bullet)^{-1})^{m}),
$$
with $\bullet= '_1, '_2, ^\sharp$ or $ ^\emptyset $, and satisfying that, on $f'^{-1}_1(Y'\setminus D^{Y'})$ and $f'^{-1}_2(Y'\setminus D^{Y'})$,
 $\bar\psi^*\bar s=\bar s'_1$ and $\bar s'_2 =\bar \eta^*\bar s^\sharp$, respectively.
 
Moreover, since all coefficients of $D^{X^\bullet}_h$ are in $(0,1)$, we can find four sections  
\begin{equation}\label{eq:4sections}
s^\bullet \in H^0(X^\bullet, (\omega_{\lceil f^\bullet\rceil}\otimes f^{\bullet*}(\cA^\bullet)^{-1})^{nm}),    
\end{equation}
for some integral $n>0$, with $\bullet= '_1, '_2, \sharp$ or $ \emptyset $, and satisfying that, on $f'^{-1}_1(Y'\setminus D^{Y'})$ and $f'^{-1}_2(Y'\setminus D^{Y'})$, $\bar\psi^* s= s'_1$ and $s'_2 = \bar\eta^* s^\sharp$, and the support of the divisor $(s^\bullet=0)$ contains the support of $D^{X^\bullet}_h$ respectively, by the definition of $\omega_{\lceil f^\bullet\rceil}$.

We now take the $nm$-th cyclic coverings given by $s^\bullet$. Then we have generic finite morphisms $\pi^\bullet:(Z^\bullet, D^{Z^\bullet}) \to (X^\bullet, D^{X^\bullet})$, with $\bullet=\emptyset, '_1, '_2, ^\sharp$, where $Z^\bullet$ is a log resolution of the main component of the normalization of the cyclic covering given by $s^\bullet$ and $D^{Z^\bullet}$ is given by $\pi^{\bullet -1}D^{X^\bullet}$.
They induce the following commutative diagram:
\begin{equation}\label{E:over Y'}
\begin{tikzcd}
    (Z, D^Z)\arrow[d, "g_0"] & (Z'_1, D^{Z'_1})\arrow[l]\arrow[dr, "g'_{01}"]\arrow[dl, phantom, "\square_1"]&(Z', D^{Z'})\arrow[l]\arrow[r]\arrow[d, "g'_0"]& (Z'_2, D^{Z'_2})\arrow[ld, "g'_{02}"]\arrow[r]\arrow[dr, phantom, "\square_2"] & (Z^\sharp, D^{Z^\sharp}) \arrow[d, "g^\sharp_0"] \\
    (Y, D^Y)& & (Y', D^{Y'})\arrow[ll, "\psi"]\arrow[rr, "\eta"] & &(Y^\sharp, D^{Y^\sharp}), 
\end{tikzcd}
\end{equation}
where $Z'$ is a common log resolution of $Z_1'$ and $Z_2'$, and $D^{Z'}$ is given by the union of the inverse image of $D^{Z'_1}$ and $D^{Z'_2}$ (so $D^{Z'}$ is a reduced divisor). We also have $\square_1$ and $\square_2$ are Cartesian over $Y \setminus D^Y$ and $Y^\sharp \setminus D^{Y^\sharp}$. Furthermore, we have sections 
$$\theta^\bullet\in H^0(\pi^{\bullet *}\omega_{\lceil f^\bullet\rceil}\otimes g_0^{\bullet*}(\cA^\bullet)^{-1}) $$ 
on $(Z^\bullet, D^{Z^\bullet})$, with $\bullet=\emptyset, '_1, '_2$ and $^\sharp$, such that their zero locus contains $D^{Z^\bullet}_h$.

\subsection{Extra boundaries and summary of the geometric construction}\label{SS:summarizing}
We make the final modification of the diagram \eqref{E:over Y'}, by adding more boundaries and summarize the geometric construction.

We continue assuming that we are in the situation of the diagram \eqref{E:over Y'}. Suppose that $f_V:(U, D^U) \to V$ is the log smooth family from where Proposition \ref{P:stable reduction} starts. Then the morphism $f^r_V:(U^r, (D_U)^r)\to V$ remains log smooth. 
Then by the construction of \eqref{E:over Y'}, we know that 
\[
\begin{tikzcd}
U^r\arrow[r,"\subseteq"]\arrow[d,"f^r_V"]& X\arrow[d,"f"]\\
V\arrow[r,"\subseteq"]& Y,
\end{tikzcd}
\]
where $f$ is the first vertical morphism (forgetting log structure) in \eqref{E:reddiagram log} in Proposition \ref{P:stable reduction log smooth} after taking the fiber product tricks as in \S \ref{SS:Product trick and cyclic covering}.
We now set $E^Y:=Y\setminus V$, $E^{X}_h$ the closure of $(D^U)^r$ in $X$ and $E^{X}_v:=f^{-1}E^Y$ and $E^{X}=E^{X}_h+E^{X}_h$. Remeber that $V$ is important, since we are interested in the log Kodaira dimension of it. We then obtain the following commutative diagram:
\begin{equation}\label{E:over Y log smooth}
\begin{tikzcd}
    & (Z, D^{Z})\arrow[d, "\pi"]\arrow[dd, bend left=60, "g_0"]\\
    (X, E^{X})\arrow[d, "f_0"]  & (X, D^X) \arrow[d, "f"]\arrow[l, "id"]\\
    (Y, E^Y) & (Y, D^Y), \arrow[l, "id"]&
\end{tikzcd}
\end{equation}
where $f_0$ is the same as $f$ as scheme-morphisms but with different log structures. 


We now make the final modification of the diagram \eqref{E:over Y'} as follows. We add extra boundary divisors on $Y^\bullet$ so that all the down morphisms are log smooth over $Y^\bullet \setminus D^{Y^\bullet}$, and both $\psi$ and $\eta$ are morphisms of pairs, with $\psi$ being finite with ramification locus contained in $D^Y$, and $\eta$ being smooth. Then we add all vertical divisors onto those $D^{Z^\bullet}$, and make further resolutions if needed, to make all down arrows are morphisms of log smooth pairs, but keeping $\square_1$ and $\square_2$ are still Cartesian over $(Y\setminus D^Y)$ and $Y^\sharp\setminus D^{Y^\sharp}$. We then use $g^\bullet$ to denote the new down morphisms respectively. After ignoring small closed subsets on those $Y^\bullet$ respectively, we can have that all of those $g^\bullet$ are log smooth, and $D^{Y^\bullet}$ are smooth divisors on $Y^\bullet$. In summary, we end up with the following commutative diagram:

\begin{equation}\label{E:over Y' log smooth}
\begin{tikzcd}
    (Z, D^Z_v)\arrow[d, "g"]& (Z'_1, D^{Z'^1}_v)\arrow[l]\arrow[dr, "g'_1"]\arrow[dl, phantom, "\square_1"]&(Z', D^{Z'}_v)\arrow[l]\arrow[r]\arrow[d, "g'"] & (Z'_2, D^{Z'_2}_v)\arrow[ld, swap, "g'_2"]\arrow[r]\arrow[dr, phantom, "\square_2"] & (Z^\sharp, D^{Z^\sharp}_v) \arrow[d, "g^\sharp"]\\
    (Y, D^Y)& & (Y', D^{Y'})\arrow[ll, "\psi"]\arrow[rr, "\eta"] & &(Y^\sharp, D^{Y^\sharp}). 
\end{tikzcd}
\end{equation}

\section{Construction of the refined Viehweg-Zuo sheaf}\label{S:refined VZ}
\subsection{Construction of the Viehweg-Zuo sheaf}
The goal of this section is the construction of the so-called Viehweg-Zuo sheaf in our setting, as indicated in the following theorem. 

We assume that we are in the situation of the diagram \eqref{E:over Y' log smooth} constructed in the previous section.
\begin{theorem}\label{T:VZ}
There exists an invertible subsheaf $\cA\subset \Omega^{\otimes n}_Y(\log E^Y)$, for some positive integer $n$, with $\kappa(\cA)\geq \Var(f_V)$.
\end{theorem}

To prove the above theorem, we also use the terminology and notations established in \S \ref{subsec:3.1} and \S \ref{SS:Direct image of graded log D-modules}. We apply Proposition \ref{P:bir invar of M} for $g_1'$ and $g'$ in the diagram \eqref{E:over Y' log smooth}, and obtain graded $\sA^\bullet_{(Y',D^{Y'})}$-modules $\hat\cM$ and $\cM_{g_1'}$.
We then prove the following lemma. 
\begin{lemma}\label{lem:seqmAm}
We have the following sequence of morphisms, as graded $\sA^\bullet_{(Y',D^{Y'})}$-modules:
\begin{equation}\label{E:seq of mor}
    \psi^*\cM_{f_0,\cL} \xrightarrow{\simeq}
\psi^*\cM_{f,\cL} \xrightarrow{}
\psi^*\cM_{g,g^*\cA} \xrightarrow{}\cM_{g'_1,g'^{*}_1\cA'}  \xrightarrow{}
 \hat{\cM}_{\cA'},
\end{equation}
away from a small closed subset of $Y$, where $\mathcal{L}=\omega_{\lceil f\rceil}=\omega_{\lceil f_0\rceil},$ $\cA^\bullet$ are defined in \eqref{eq:3lbdle}, and $\hat{\cM}_{\cA'}:= \hat\cM\otimes_{\cO_{Y'}} \cA'^{-1}$. Moreover, the lowest graded pieces of those morphisms are inclusions. In particular, the composition of those morphisms is not trivial. 
\end{lemma}
\begin{proof}
We construct \eqref{E:seq of mor} one by one from the left to the right. By \eqref{E:id for extra boundary}, we obtain the first isomorphism. To obtain the second morphism, we need the construction in \S \ref{SS:Sub higgs sheaves from a section}. More precisely, by \eqref{eq:3lbdle} and the construction of $Z$ in the \S \ref{SS:Product trick and cyclic covering}, we see that (Assumption. 1) and (Assumption. 2) in \S \ref{SS:Sub higgs sheaves from a section} are fulfilled. Taking $\theta=s$ there in \S \ref{SS:Sub higgs sheaves from a section}, we see that the second morphism is induced by $\eta_s$. We apply projection formula and see that 
\[\cM_{g,g^*\cA}=\cM_{g}\otimes_\sO \cA \textup{ and } \cM_{g_1',g_1'^*\cA'}=\cM_{g_1'}\otimes_\sO \cA'.\]
We then obtain the third morphism by Proposition \ref{P:morphism due to base change}. By Proposition \ref{P:bir invar of M}, we know that $\hat\cM$ is a direct summand of $\cM_{g_1'}$. Hence, the last morphism is defined to be the natural projection 
\[\cM_{g'_1,g'^{*}_1\cA'}  \xrightarrow{}
 \hat\cM_{\cA'}.\]
The second morphism on the lowest graded pieces is injective by construction. The third morphism is injective. The last morphism on the lowest graded pieces is identical thanks to Proposition \ref{P:bir invar of M} again. Therefore, the lowest graded pieces of those morphisms are inclusions.
\end{proof}



Denote by $\cN_{g'}$, the image of $\psi^*\cM_{f_0,\cL}$ in $\hat{\cM}_{\cA'}$, under the morphisms in \eqref{E:seq of mor}. Note that it is naturally a sub-$\sA^\bullet_{(Y', D^{Y'})}$-module of $\hat{\cM}_{\cA'}$. However, since $\cM_{f_0,\cL}$ is a $\sA^\bullet_{(Y, E^Y)}$-module, we have the following commutative diagram,
\begin{equation*}
    \begin{tikzcd}
    \psi^*\cM_{f_0,\cL}\arrow[r]\arrow[d]& \psi^*\cM_{f_0,\cL}\otimes \psi^*\Omega_Y(\log E^Y)\arrow[d]\\
    \hat{\cM}_{\cA'} \arrow[r] & \hat{\cM}_{\cA'}\otimes  \Omega_{Y'}(\log D^{Y'}).
    \end{tikzcd}
\end{equation*}
Consequently, the Higgs-type morphism $\cN_{g'}\to \cN_{g'} \otimes \Omega_{Y'}(\log D^{Y'})$ factorizes through 
\begin{equation}\label{E:correct log pole}
    \delta: \cN_{g'}\to \cN_{g'} \otimes \psi^*\Omega_Y(\log E^Y).
\end{equation}
We denote by $\cC_{g'}$ the kernel of $\delta$ and denoted by 
\[\delta^k: \cN_{g'}\to \cN_{g'} \otimes \psi^*\Sym^k\Omega_Y(\log E^Y),\]
the $k$-th iterated morphism of $\delta$.

\begin{prop}\label{P:gg for -N-A}
For any positive integer $m$, we can find another positive integer $N_m$, such that
$\psi^*\cA^{-m}\otimes \Sym^{N_m}\cC_{g'}^\vee$ is generically globally generated.
\end{prop}

\begin{proof}
We first focus on the morphism $f^\sharp$ in \eqref{E:reddiagram log} in Proposition \ref{P:stable reduction log smooth} after taking the fiber product tricks as in \S \ref{SS:Product trick and cyclic covering}. We apply Theorem \ref{P:qe for -N-L} and get that $\mathcal{K}^\vee_{g^\sharp}$ is weakly positive. Since $\cA^\sharp$ is big, we hence have that $$\mathcal{K}^\vee_{g^\sharp,g^{\sharp*}\cA^\sharp}\simeq \mathcal{K}^\vee_{g^\sharp}\otimes \cA^\sharp$$ 
is also big. By the definition of bigness (see for instance \cite[Definition 4.7(2)]{KP16}), we then have that for any integer $m$, there exists an integer $N_m$, such that
$$\cA^{\sharp-m}\otimes \Sym^{N_m}\cK^\vee_{g^\sharp,g^{\sharp*}\cA^\sharp}$$ 
is generically globally generated. Since $\eta$ is smooth, we have $\eta^*\cM_{g^\sharp}\simeq \cM_{g'_2}$ and hence  $\eta^* \cK_{g^\sharp}\simeq \cK_{g'_2}.$ Since $\eta^*\cA^{\sharp}\simeq \cA'\simeq \psi^*\cA$ (see \eqref{eq:3lbdle}), 
we hence have an isomorphism
$$\eta^*(\cA^{\sharp-m}\otimes \Sym^{N_m}\cK^\vee_{g^\sharp,g^{\sharp*}\cA^\sharp})\to \psi^*\cA^{-m}\otimes \Sym^{N_m}\cK^\vee_{g'_2,g^{'*}_2\cA'}.$$
Therefore, $\psi^*\cA^{-m}\otimes \Sym^{N_m}\cK^\vee_{g'_2,g^{'*}_2\cA'}$ is also generically globally generated. 

By applying Proposition \ref{P:bir invar of M} for $g'$ and $g'_2$, we see that $\hat{\cM}_{\cA'}$ is a direct summand of $\cM_{g'_2,g'^{2*}\cA'}$. We have a natural inclusion $\cC_{g'}\to \cK_{g'_2,g^{*}_2\cA'}$, which induces a morphism 
$$\psi^*\cA^{-m}\otimes \Sym^{N_m}\cK^\vee_{g'_2,g'^{*}_2\cA'}\to \psi^*\cA^{-m}\otimes \Sym^{N_m}\cC_{g'}^\vee,$$
which is generically surjective. Hence we can conclude the proof.


\end{proof}

The following techinical lemma is needed in the proof of Theorem \ref{T:VZ}. It is well-know by experts, but we still give a proof here.
\begin{lemma}\label{lm:decffinite}
Given a quasi-coherent sheaf $\cF$ on a variety $X$, and a finite covering $f:X'\to X$, assuming that $f^*\cF$ has a non-trivial global section, then there exists a positive integer $n$, such that $H^0(X, \cF^{\otimes n})\neq 0.$ 
\end{lemma}
\begin{proof}
Let $\bar f:\bar X\to X$ be the Galois closure of $f$, with Galois group $G$ of order $n$. Pick any non-trivial global section $s$ in $H^0(\bar X, \bar f^*\cF)$. Then $\bigotimes_{g\in G}g\cdot s$ gives a $G$-invariant global section of $H^0(\bar X, \bar f^*\cF^{\otimes n})$, and hence induces a global section of $\cF^{\otimes n}$.
\end{proof}

\begin{proof}[Proof of Theorem \ref{T:VZ}]
By construction, the lowest graded piece of $\cM_{f_0,\cL}$ is $f_*\cO_X$, which has a non-trivial global section (generating $\sO_Y$), and denote by $t$ its image in $\cN_{g'}$. By the inclusions obtained in Lemma \ref{lem:seqmAm}, we have that $t$ is not a trivial section. Let $k$ be the largest number, such that $\delta^k$ does not map $t$ to $0$. By Proposition \ref{P:gg for -N-A}, we have that $k\geq 1$. 
Hence $\delta^k$ maps $t$ to $\cC_{g'}\otimes \psi^*\Sym^k \Omega_Y(\log E^Y)$,
which induces a non-trivial morphism 
$$\cC_{g'}^\vee \to \psi^*\Sym^k \Omega_Y(\log E^Y).$$
Since $\psi^*\cA^{-1}\otimes \Sym^m\cC_{g'}^\vee$ is generically globally generated, for some $m$, we can find a non-trivial morphism 
$$\psi^*\cA \to \psi^*\Sym^{mk} \Omega_Y(\log E^Y).$$
Since $\psi$ is finite, we can find a positive integer $n$, such that we have a non-trivial morphism
$$\cA^n \to \Sym^{nmk} \Omega_Y(\log E^Y),$$
thanks to Lemma \ref{lm:decffinite}.
Since $\sA$ is a line bundle, the above non-trivial morphism is an inclusion. 
We conclude the proof, by replacing $\cA$ by $\cA^n$. 
\end{proof}

\subsection{Refinement of Viehweg-Zuo sheaves}

We first denote 
$$\cQ_{(Y,D^Y)}:=\ker(\cT_Y(-\log D^Y)\to \cB^\vee_{(Y,D^Y)}).
$$
For a graded $\mathscr{A}^\bullet_{(Y,D^Y)}$-module $\cM$, to show that the Higgs-type morphism 
$$\cM \to \cM \otimes \Omega_Y(\log D^Y)$$
factors through 
$$\cM\to \cM \otimes \cB_{(Y,D^Y)},$$
we only need to show that $\cQ_{(Y,D^Y)}$ acts on $\cM$ trivially.
If we further have $\cM$ being $\cO_Y$-torsion-free, we only need to show that it acts trivially, over a dense open subset of $Y$. 

\begin{proof}[Proof of Theorem \ref{T:refined VZ}]
We first show that the higgs morphism (\ref{E:correct log pole})
$$\delta: \cN_{g'}\to \cN_{g'} \otimes \psi^*\Omega_Y(\log E^Y).
$$
factors through 
\begin{equation*}
    \delta_0: \cN_{g'}\to \cN_{g'} \otimes \psi^*\cB_{(Y, E^{Y})}.
\end{equation*}
Using the two stable families $f'_{2,c}$ and $f^\sharp_c$ constructed in Proposition \ref{P:stable reduction}, over $Y'$ and $Y^\sharp$, respectively, we have the induced moduli maps $\mu^\bullet: Y^\bullet\to \mathfrak{M}$, with $\bullet=', \sharp$. Hence, following Definition \ref{D:def of B}, we can similarly define $\cB_{(Y',D^{Y'})}$ and  $\cB_{(Y^\sharp,D^{Y^\sharp})}$, as well as $\cQ_{(Y', D^{Y'})}$ and $\cQ_{(Y^\sharp, D^{Y^\sharp})}$ as above.

Note that, since $\psi$ is finite, over a dense open subset of $Y'$, we have $\psi^*\cQ_{(Y, E^{Y})}= \cQ_{(Y', D^{Y'})}$.
Hence, we only need to show that $\cQ_{(Y', D^{Y'})}$ acts on $\cN_{g'}$ trivially, over some dense open subset of $Y'$. 

The isomorphism $\eta^*\cM_{g^\sharp,g^{\sharp*}\cA^\sharp}\simeq \cM_{g'_2,g'^{*}_2\cA'}$ implies that the kernel of the natural morphism 
$d\eta^*:\cT_{Y'}(-\log D^{Y'})\to \eta^*\cT_{Y^\sharp}(-\log D^{Y^\sharp})$ acts trivially on $\cM_{g'_2,g'^{*}_2\cA'}$, hence on $\cN_{g'}$. 
Since $f^\sharp$ is of maximal variation, we have $\Omega_{Y^\sharp}(\log D^{Y^\sharp})=\cB_{(Y^\sharp,D^{Y^\sharp})}$, and that the inclusion $\eta^*\Omega_{Y^\sharp}(\log D^{Y^\sharp})\subset \cB_{(Y',D^{Y'})}$ actually is an isomorphism, over a dense open subset of $Y'$. By taking the dual, we have that, over a dense open subset, $\cQ_{(Y',D^{Y'})} $ is isomorphic to the kernel of $d\eta^*$, and hence it acts trivially on $\cN_{g'}$.

Furthermore, we have that $\cC_{g'}$, the kernel of $\delta$, is also the kernel of $\delta_0$, since the induced morphism 
$$ \cN_{g'} \otimes \psi^*\cB_{(Y, E^{Y})}\to \cN_{g'} \otimes \psi^*\Omega_Y(\log E^Y)$$
is an inclusion, thanks to the fact that $\cN_{g'} \otimes \psi^*\cB_{(Y, E^{Y})}$ is torsion-free, as $\psi$ is flat. 

Now, we can apply the argument in the proof of Theorem \ref{T:VZ},
by replacing $\Omega_Y(\log E^Y)$ by $\cB_{(Y,E^Y)}$, and $\delta$ by $\delta^0$, from which, the proof is accomplished.
\end{proof}

\bibliographystyle{alpha}
\bibliography{mybib}
\end{document}